\documentclass[12pt,reqno]{amsart}
\usepackage{amssymb,amsmath,hyperref,amsthm}
\newtheorem{theorem}{Theorem}

\newtheorem{corollary}[theorem]{Corollary}
\newtheorem{proposition}[theorem]{Proposition}

\theoremstyle{remark}

\theoremstyle{definition}

\numberwithin{theorem}{section} 
\numberwithin{equation}{section}
\numberwithin{example}{section}

\title{On ranks and cranks of partitions modulo $4$ and $8$}

\author{Eric T. Mortenson}

\begin{document}

\date{31 July 2017}



\begin{abstract}  
Denote by $p(n)$ the number of partitions of $n$ and by $N(a,M;n)$ the number of partitions of $n$ with rank congruent to $a$ modulo $M$.  By considering the deviation
\begin{equation*}
D(a,M) := \sum_{n= 0}^{\infty}\left(N(a,M;n) - \frac{p(n)}{M}\right) q^n,
\end{equation*}
we give new proofs of recent results of  Andrews, Berndt, Chan, Kim and Malik on mock theta functions and ranks of partitions.  
By considering deviations of cranks, we give new proofs of Lewis and Santa-Gadea's rank-crank identities.  We revisit ranks and cranks modulus $M=5$ and $7$, with our results on cranks appearing to be new.   We also demonstrate how considering deviations of ranks and cranks gives first proofs of Lewis's conjectured identities and inequalities for rank-crank differences of modulus $M=8$.   
\end{abstract}

\email{etmortenson@gmail.com}
\maketitle
\setcounter{section}{-1}

\section{Notation}
Let $q$ be a complex number with $0<|q|<1$ and define $\mathbb{C}^*:=\mathbb{C}-\{0\}$.  We recall:
 \begin{gather*}
(x)_n=(x;q)_n:=\prod_{i=0}^{n-1}(1-q^ix), \ \ (x)_{\infty}=(x;q)_{\infty}:=\prod_{i\ge 0}(1-q^ix),\\
 {\text{and }}\ \ j(x;q):=(x)_{\infty}(q/x)_{\infty}(q)_{\infty}=\sum_{n}(-1)^nq^{\binom{n}{2}}x^n,
\end{gather*}
where in the last line the equivalence of product and sum follows from Jacobi's triple product identity.  Let $a$ and $m$ be integers with $m$ positive.  Define
\begin{gather*}
J_{a,m}:=j(q^a;q^m), \ \ \overline{J}_{a,m}:=j(-q^a;q^m), \ {\text{and }}J_m:=J_{m,3m}=\prod_{i\ge 1}(1-q^{mi}).
\end{gather*}

\section{Introduction}
We recall a universal mock theta function
\begin{equation}
g(x;q):=x^{-1}\Big ( -1 +\sum_{n=0}^{\infty}\frac{q^{n^2}}{(x)_{n+1}(q/x)_{n}} \Big ).\label{equation:g-def}
\end{equation}

One of the earliest celebrated results in the history of mock theta functions was Hickerson's proof of the mock theta conjectures, that express fifth order mock theta functions $f_0(q)$ and $f_1(q)$ in terms of the universal mock theta function $g(x;q)$:
\begin{theorem} \cite{H1} The following identities are true:
\begin{align*}
f_0(q)&:=\sum_{n=0}^{\infty}\frac{q^{n^2}}{(-q;q)_n}=-2q^2g(q^2;q^{10})+\frac{J_{5,10}J_{2,5}}{J_1},\\
f_1(q)&:=\sum_{n=0}^{\infty}\frac{q^{n^2+n}}{(-q;q)_n}=-2q^3g(q^4;q^{10})+\frac{J_{5,10}J_{1,5}}{J_{1}}.
\end{align*}
\end{theorem}

Mock theta functions and the study of partitions are inextricably linked.  A partition of a positive integer $n$ is a weakly-decreasing sequence of positive integers whose sum is $n$.  For example the partitions of the number $4$ are $(4)$, $(3,1)$, $(2,2)$, $(2,1,1)$, $(1,1,1,1)$.  We denote the number of partitions of $n$ by $p(n)$.   Among the most famous results in the theory of partitions are Ramanujan's congruences:
\begin{subequations}
\begin{equation}
p(5n+4)\equiv 0 \pmod 5,\notag
\end{equation}
\begin{equation}
p(7n+5)\equiv 0 \pmod 7,\notag
\end{equation}
\begin{equation}
p(11n+6)\equiv 0 \pmod {11}.\notag
\end{equation}
\end{subequations}

To study Ramanujan's partition congruences, Dyson constructed a function which assigns an integer value to a partition.  Dyson defined the {\em rank of a partition} to be the largest part minus the number of parts.  As an example, the ranks of the five partitions of $4$ are $3, 1, 0, -1, -3$, respectively, giving an equinumerous distribution of the partitions of $4$ into the five residue classes mod $5$.    We further define
\begin{equation*}
N(a,M;n) := \mathrm{number\ of\ partitions\ of\ } n \mathrm{\ with\ rank} \equiv a \ (\mathrm{mod\ } M),
\end{equation*} 
which has the symmetric property $N(a,M,n)=N(M-a,M;n)$.   To explain Ramanujan's first two congruences, Dyson conjectured and Atkin and Swinnerton-Dyer proved \cite{ASD, D}
\begin{align*}
N(a,5;5n+4)&=p(5n+4)/5,   \ \textup{for}\  0\le a\le4, \\
N(a,7;7n+5)&=p(7n+5)/7,  \ \textup{for}\  0\le a\le6.
\end{align*}
For more identities on ranks modulo $M=5$ or $7$ see \cite{D}, \cite[$(2.2)$--$(2.11)$]{ASD}.   For analogous results for other low moduli $M$, see \cite{LSG, NSG2}.

Although the rank does not explain Ramanujan's third congruence, Dyson conjectured another function, which he called the {\em crank}, that would divide the partitions of $11n+6$ into eleven equal classes.  Andrews and Garvan later discovered the crank \cite{AG}.   For a partition $\pi$, let $\lambda(\pi)$ denote the largest part, $\nu(\pi)$ the number of ones, and $\mu(\pi)$ the number of parts larger than $\nu(\pi)$.  The crank of $\pi$, denoted $c(\pi)$, is defined as follows
\begin{align*}
c(\pi):=\begin{cases}
\lambda(\pi), & \textup{when }  \nu(\pi)=0,\\
\mu(\pi)-\nu(\pi), & \textup{otherwise}.
\end{cases}
\end{align*}
The cranks of the five partitions of $4$ are $4,0,2,-2,-4$, respectively, giving an equinumerous distribution of the partitions of $4$ into the five residue classes mod $5$.  Defining
\begin{equation*}
C(a,M;n) := \mathrm{number\ of\ partitions\ of\ } n \mathrm{\ with\ crank} \equiv a \ (\mathrm{mod\ } M),
\end{equation*} 
Andrews and Garvan showed
\begin{align*}
C(a,5;5n+4)&=p(5n+4)/5,  \ \textup{for}\  0\le a\le4, \\
C(a,7;7n+5)&=p(7n+5)/7, \ \textup{for}\ 0\le a\le6,\\
C(a,11;11n+6)&=p(11n+6)/11, \ \textup{for}\ 0\le a\le10.
\end{align*}

Ranks and cranks are related. We point out that we do not consider ranks or cranks of the partition of zero.  Lewis and Santa-Gadea proved identities such as \cite[$(8)$--$(17)$]{LSG}:
\begin{subequations}
\begin{equation}
N(2,4;2n)=C(1,4;2n),\label{equation:NC-8}
\end{equation}
\begin{equation}
N(0,4;2n+1)=C(1,4;2n+1),\label{equation:NC-9}
\end{equation}
\begin{equation}
C(1,8;4n)=C(3,8;4n)=N(2,8;4n)=N(4,8;4n),\label{equation:NC-10}
\end{equation}
\begin{equation}
\begin{split}
C&(0,8;4n+1)+C(1,8;4n+1)
=C(3,8;4n+1)+C(4,8;4n+1)\\
&=N(1,8;4n+1)+N(2,8;4n+1)
=N(3,8;4n+1)+N(4,8;4n+1),\label{equation:NC-11}
\end{split}
\end{equation}
\begin{equation}
C(1,8;4n+2)
=C(3,8;4n+2)
=N(0,8;4n+2)
=N(2,8;4n+2),\label{equation:NC-12}
\end{equation}
\begin{equation}
\begin{split}
C&(0,8;4n+3)+C(1,8;4n+3)
=C(3,8;4n+3)+C(4,8;4n+3)\\
&=N(0,8;4n+3)+N(1,8;4n+3)
=N(2,8;4n+3)+N(3,8;4n+3),\label{equation:NC-13}
\end{split}
\end{equation}
\begin{equation}
N(3,8;4n)=C(2,8;4n),\label{equation:NC-14}
\end{equation}
\begin{equation}
N(3,8;4n+1)=C(2,8;4n+1),\label{equation:NC-15}
\end{equation}
\begin{equation}
N(1,8;4n+2)=C(2,8;4n+2),\label{equation:NC-16}
\end{equation}
\begin{equation}
N(1,8;4n+3)=C(2,8;4n+3).\label{equation:NC-17}
\end{equation}
\end{subequations}

Andrews, Berndt, Chan, Kim and Malik \cite{ABCKM} recently proved results on mock theta functions and partitions and found results analogous to works of Dyson and Atkin and Swinnerton-Dyer but for modulus $M=4$ and $8$.  They showed  \cite[$(7.5)$, $(7.6)$]{ABCKM}
\begin{align}
N(0,4;2n)-N(2,4;2n)&=(-1)^n\big [ N(0,8;2n)-N(4,8;2n)\big ],\label{equation:ABCKM-id7.5}\\
 N(0,4;2n+1)-N(2,4;2n+1)&=(-1)^n\big [ N(0,8;2n+1)+N(1,8;2n+1)\label{equation:ABCKM-id7.6}\\
 &\ \ \ \ \ \ \ \ \ \ -2N(3,8;2n+1)-N(4,8;2n+1)\big ].\notag 
\end{align}

\noindent Identities (\ref{equation:ABCKM-id7.5}) and (\ref{equation:ABCKM-id7.6}) follow from their two (slightly rewritten) main theorems:
\begin{theorem}\cite[Theorem $1.6$]{ABCKM}  \label{theorem:ABCKM-Thm1.6} We have
\begin{align}
\sum_{n=0}^{\infty}\Big ( N(0,4;n)&-N(2,4;n)\Big )q^{n}\label{equation:ABCKM-Thm1.6} \\
&= 2-2q^2g(-q^2;q^{16}) +2q^5g(-q^6;q^{16})
-\frac{J_{2,4}\overline{J}_{6,16}}{J_4}+q\frac{J_{2,4}\overline{J}_{2,16}}{J_4}.\notag
\end{align}
\end{theorem}
\begin{theorem} \cite[Theorem $1.7$]{ABCKM} \label{theorem:ABCKM-Thm1.7} We have
\begin{align}
\sum_{n= 0}^{\infty}\Big ( N(0,8;n)-N(4,8;n)\Big )q^{n} 
&=2+2q^2g(q^2;q^{16})-\frac{\overline{J}_{2,4}  J_{6,16}}{J_4}
+q\frac{\overline{J}_{2,4}J_{2,16}}{J_4}, \label{equation:ABCKM-Thm1.7A} \\
\sum_{n= 0}^{\infty}\Big ( N(1,8;n)-N(3,8;n)\Big )q^{n}
&=-1-q^2g(q^2;q^{16})+q^5g(q^6;q^{16})
+\frac{\overline{J}_{2,4}J_{6,16}}{J_4}. \label{equation:ABCKM-Thm1.7B}
\end{align}
\end{theorem}

In this note we will demonstrate how methods and results from our work on mock theta functions and Dyson's ranks \cite{HM2,M1} can be used to prove results on ranks and cranks of partitions such as those found in \cite{ABCKM, LSG, L1}. We define the deviation of the ranks from the expected value as
\begin{equation}
D(a,M) = D(a, M;q) := \sum_{n= 0}^{\infty}\left(N(a,M;n) - \frac{p(n)}{M}\right) q^n,\label{equation:rd-def}
\end{equation}
and the deviation of the cranks from the expected value as
\begin{equation}
D_C(a,M)=D_C(a,M;q) := \sum_{n=0}^{\infty}\left(C(a,M;n) - \frac{p(n)}{M}\right) q^n. \label{equation:cd-def}
\end{equation}

By determining the dissections of the relevant deviations it becomes straightforward to prove identities such as
 $(\ref{equation:ABCKM-Thm1.6})$--$(\ref{equation:ABCKM-Thm1.7B})$ and $(\ref{equation:NC-8})$--$(\ref{equation:NC-17})$.   For a preview, using notation
\begin{equation}
\begin{split}
\vartheta_4&(a_0,a_1,a_2,a_3)\\
&:=\frac{1}{4J_4}\Big [ 
a_0\cdot \overline{J}_{4,8}\overline{J}_{6,16}
+a_1\cdot q^2 \overline{J}_{0,8}\overline{J}_{14,16}
+a_2\cdot q  \overline{J}_{4,8}\overline{J}_{14,16}
+a_3\cdot q  \overline{J}_{0,8}\overline{J}_{6,16}\Big ], 
\end{split}
\end{equation}
and
\begin{equation}
G_4(b_0,b_1):=b_0\cdot (-1+q^2g(-q^2;q^{16}))+b_1\cdot q^5g(-q^6;q^{16}),
\end{equation}
we will show the following two theorems which will prove $(\ref{equation:ABCKM-Thm1.6})$, $(\ref{equation:NC-8})$, and $(\ref{equation:NC-9})$.
\begin{theorem} \label{theorem:D4-deviants} We have the following $2$-dissections:
\begin{align}
D(0,4)&=\vartheta_4(-5,3,1,1)+ 2\cdot G_4(-1,0),\label{equation:D04-deviant-alt}\\ 
D(1,4)=D(3,4)&= \vartheta_4(3,-1,-3,1)+G_4(1,1), \label{equation:D14-deviant-alt}\\
D(2,4)&=\vartheta_4(-1,-1,5,-3)+2\cdot G_4(0,-1).\label{equation:D24-deviant-alt}
\end{align}
\end{theorem}
\begin{theorem}\label{theorem:DC4-deviants} We have the following $2$-dissections:
{\allowdisplaybreaks \begin{align}
D_C(0,4)&=\vartheta_4(3,-1,1,-3),\label{equation:DC04-deviant-final}\\
D_C(1,4)=D_C(3,4)&=\vartheta_4(-1,-1,1,1),\label{equation:DC14-deviant-final}\\
D_C(2,4)&=\vartheta_4(-1,3,-3,1).\label{equation:DC24-deviant-final}
\end{align}}%
\end{theorem}

Identity $(\ref{equation:ABCKM-Thm1.6})$ follows from evaluating the difference $D(0,4)-D(2,4)$.   Given $(\ref{equation:g-def})$, we see that the Fourier expansions of $g(- q^{2};q^{16})$ and $g( -q^{6};q^{16})$ are supported on even powers of $q$.  Hence $b_0, a_0, a_1$ are coefficients of even $q$-powers and $b_1, a_2, a_3$ are coefficients of odd $q$-powers.  Comparing even powers of $q$ in $(\ref{equation:DC14-deviant-final})$ and $(\ref{equation:D24-deviant-alt})$ gives the first rank-crank relation $(\ref{equation:NC-8})$.  Relation $(\ref{equation:NC-9})$ follows from comparing odd powers of $q$ in $(\ref{equation:DC14-deviant-final})$ and $(\ref{equation:D04-deviant-alt})$.   In each theorem, the  deviations sum to zero.

In Section \ref{section:prelim}, we cover preliminaries.  In Section \ref{section:DM4-table}, we prove Theorem \ref{theorem:D4-deviants}.   In Section \ref{section:ABCKM-Thm1.6}, we prove Theorem \ref{theorem:ABCKM-Thm1.6}.  In Sections \ref{section:D8-prelim} and \ref{section:ABCKM-Thm1.7}, we prove Theorem \ref{theorem:ABCKM-Thm1.7}.   In Section \ref{section:DC4-deviants}, we prove Theorem \ref{theorem:DC4-deviants}.   In Section \ref{section:crank-prelims} and Section \ref{section:rank-crank}, we use our methods to prove identities $(\ref{equation:NC-10})$--$(\ref{equation:NC-17})$.  In Section \ref{section:rankscranks57}, we rewrite the dissections of rank deviations modulus $M=5$ and $7$ and obtain corresponding dissections for cranks, which appear to be new.  In the final section, we demonstrate how our results prove Lewis's conjectured identities and inequalities for rank-crank differences of modulus $8$ \cite{L1}.

\section{Preliminaries}\label{section:prelim}

For later use, we list useful product rearrangements:
{\allowdisplaybreaks \begin{gather*}
\overline{J}_{0,1}=2\overline{J}_{1,4}=\frac{2J_2^2}{J_1},\   \overline{J}_{1,2}=\frac{J_2^5}{J_1^2J_4^2}, \   J_{1,2}=\frac{J_1^2}{J_2},  \  \overline{J}_{1,3}=\frac{J_2J_3^2}{J_1J_6}, \\
 J_{1,4}=\frac{J_1J_4}{J_2}, \   J_{1,6}=\frac{J_1J_6^2}{J_2J_3},  \  \overline{J}_{1,6}=\frac{J_2^2J_3J_{12}}{J_1J_4J_6}.
\end{gather*}}%

\noindent We recall more theta function identities, here $\zeta_n$ is a primitive $n$-root of unity:
{\allowdisplaybreaks \begin{subequations}
\begin{equation}
j(qx^3;q^3)+xj(q^2x^3;q^3)={J_1j(x^2;q)}/{j(x;q)},\label{equation:H1Thm1.0}
\end{equation}
\begin{equation}
j(q x;q)=-x^{-1}j(x;q),\label{equation:1.8}
\end{equation}
\begin{equation}
j(x;q)=j(q/x;q),\label{equation:1.7}
\end{equation}
\begin{equation}
j(x;q)={J_1}j(x;q^2)j(qx;q^2)/{J_2^2},\label{equation:1.10}
\end{equation}
\begin{equation}
j(x;-q)={j(x;q^2)j(-qx;q^2)}/{J_{1,4}},\label{equation:1.11}
\end{equation}
\begin{equation}
j(z;q)=\sum_{k=0}^{m-1}(-1)^k q^{\binom{k}{2}}z^k
j\big ((-1)^{m+1}q^{\binom{m}{2}+mk}z^m;q^{m^2}\big ),\label{equation:jsplitgen}\\
\end{equation}
\begin{equation}
j(x^n;q^n)={J_n}j(x,\zeta_nx,\dots,\zeta_n^{n-1}x;q^n)/{J_1^n} \ \ {\text{if $n\ge 1$.}}\label{equation:1.12}
\end{equation}
\end{subequations}}%
\noindent Identity (\ref{equation:H1Thm1.0}) is the quintuple product identity.   A frequently used form of $(\ref{equation:jsplitgen})$ reads
\begin{subequations}
\begin{gather}
j(z;q)=j(-qz^2;q^4)-zj(-q^3z^2;q^4),\label{equation:jsplit}
\end{gather}
\end{subequations}

\begin{proposition} \cite[Theorems $1.1$-$1.2$]{H1}  For generic $x,y\in \mathbb{C}^*$ 
\begin{subequations}
\begin{equation}
j(x;q)j(y;q)=j(-xy;q^2)j(-qx^{-1}y;q^2)-xj(-qxy;q^2)j(-x^{-1}y;q^2),\label{equation:H1Thm1.1}
\end{equation}
\begin{equation}
j(-x;q)j(y;q)+j(x;q)j(-y;q)=2j(xy;q^2)j(qx^{-1}y;q^2).\label{equation:H1Thm1.2B}
\end{equation}
\end{subequations}
\end{proposition}
We recall the three-term Weierstrass relation for theta functions \cite[(1.)]{We}:
\begin{proposition}\label{proposition:Weierstrass-id} For generic $a,b,c,d\in \mathbb{C}^*$
\begin{equation}
j(ac,a/c,bd,b/d;q)=j(ad,a/d,bc,b/c;q)+b/c \cdot j(ab,a/b,cd,c/d;q).\label{equation:Weierstrass}
\end{equation}
\end{proposition}

We recall a fact which follows from \cite[Lemma $2$]{ASD} and is also \cite[Theorem $1.7$]{H1}.
\begin{proposition}\label{proposition:H1Thm1.7} Let $C$ be a nonzero complex number, and let $n$ be a nonnegative integer.  Suppose that $F(z)$ is analytic for $z\ne 0$ and satisfies $F(qz)=Cz^{-n}F(z)$.  Then either $F(z)$ has exactly $n$ zeros in the annulus $|q|<|z|\le 1$ or $F(z)=0$ for all $z$.
\end{proposition}

\begin{proposition}\cite[Proposition $3.4$]{M1}\label{proposition:hecke-phi-id}  Let $x\ne0.$  Then
\begin{equation}j(q^2x;q^4) j(q^5x;q^8)+\frac{q}{x}\cdot j(x;q^4) j(qx;q^8)-\frac{J_1}{J_4}\cdot j(-q^3x;q^4)j(q^3x;q^8)=0.
\label{equation:hecke-phi-id}
\end{equation}
\end{proposition}

Using Proposition \ref{proposition:H1Thm1.7}, we can prove a result similar to (\ref{equation:hecke-phi-id}).
\begin{proposition}\label{proposition:hecke-phi-id-alt}  Let $x\ne0.$  Then
\begin{equation}j(-x;q^4) j(-q^5x;q^8)- j(-q^2x;q^4) j(-qx;q^8)-x\frac{J_1}{J_4}\cdot j(q^3x;q^4)j(-q^7x;q^8)=0.
\label{equation:hecke-phi-id-alt}
\end{equation}
\end{proposition}
\begin{proof}[Proof of Proposition \ref{proposition:hecke-phi-id-alt}]  Let $f(x)$ be the left-hand side of (\ref{equation:hecke-phi-id-alt}).  We have $f(q^8x)= q^{-9}x^{-3}f(x).$  By Proposition \ref{proposition:H1Thm1.7}, if $f$ has more than three zeros in $|q^8|<|x|\le 1$, then $f(x)=0$ for all $x\ne 0$.  But it is easy to check that $f(x)=0$ for $x=-1, -q, -q^2, -q^3$.
\end{proof}

The following identity will be our workhorse and can be found in the lost notebook:
\begin{proposition}{\cite[p. $32$]{RLN}, \cite[$(12.5.3)$]{ABI}} \label{proposition:gsplit} For generic $x\in \mathbb{C}$
\begin{equation}
g(x;q)=-x^{-1}+qx^{-3}g(-qx^{-2};q^4)-qg(-qx^2;q^4)+\frac{J_2J_{2,4}^2}{xj(x;q)j(-qx^2;q^2)}.\label{equation:gsplit}
\end{equation}
\end{proposition}
\noindent Proposition \ref{proposition:gsplit} has a useful and easily shown corollary:
\begin{corollary} {\cite[p. $39$]{RLN}, \cite[$(12.4.4)$]{ABI}} \label{corollary:rootsof1} For generic $x\in \mathbb{C}$
\begin{subequations}
\begin{equation}
g(x;q)+g(-x;q)=-2qg(-qx^2;q^4)+\frac{2J_2\overline{J}_{1,4}^2}{j(-qx^2;q^4)j(x^2;q^2)},\label{equation:rootsof1n2k0}
\end{equation}
\begin{equation}
g(x;q)-g(-x;q)=-2x^{-1}+2qx^{-3}g(-qx^{-2};q^4)+\frac{2J_2\overline{J}_{1,4}^2}{xj(-q^3x^2;q^4)j(x^2;q^2)}.\label{equation:rootsof1n2k1}
\end{equation}
\end{subequations}
\end{corollary}

Let us denote by $N(m,n)$ the number of partitions of $n$ with rank equal to $m$.  The generating function for $N(m,n)$ is given by
\begin{equation}
\sum_{n=0}^{\infty}\sum_{m=-\infty}^{\infty}N(m,n)z^mq^n=\sum_{n=0}^{\infty}\frac{q^{n^2}}{(zq)_n(z^{-1}q)_n}.\label{equation:N-generating-fn}
\end{equation}
Rank deviations $(\ref{equation:rd-def})$ can be computed using $(\ref{equation:N-generating-fn})$ and $(\ref{equation:g-def})$:
\begin{equation}
D(a,M) =\frac{1}{M}\sum_{j=0}^{M-1}\zeta_M^{-aj}\Big (1-\zeta_M^j\Big )\Big (1+\zeta_M^jg(\zeta_M^j;q)\Big ), \label{equation:rankdeviant-def}
\end{equation}
where $\zeta_M$ is a primitive $M$-th root of unity.  In general, one expresses $g(x;q)$ in terms of Appell--Lerch functions and then sums them over roots of unity using \cite[Theorem $3.9$]{HM1}, see \cite{HM2}.  In our setting the modulus $M$ is a power of two, so we use instead Corollary \ref{corollary:rootsof1}.

Let us denote by $C(m,n)$ the number of partitions of $n$ with crank equal to $m$.  The generating function for $C(m,n)$ is given by
\begin{equation}
\sum_{n=0}^{\infty}\sum_{m=-\infty}^{\infty}C(m,n)z^mq^n=\prod_{n=1}^{\infty}\frac{(1-q^n)}{(1-zq^n)(1-z^{-1}q^n)},
\end{equation}
and the analogous deviation from the expected value is
\begin{equation}
D_C(a,M) =\frac{1}{M}\sum_{j=1}^{M-1}\zeta_M^{-aj}\frac{(q)_{\infty}}{(\zeta_M^jq)_{\infty}(\zeta_M^{-j}q)_{\infty}}, 
\label{equation:crankdeviant-def}
\end{equation}
where $\zeta_M$ is a primitive $M$-th root of unity.

\section{Proof of Theorem \ref{theorem:D4-deviants}} \label{section:DM4-table}

Theorem \ref{theorem:D4-deviants} is a straightforward consequence of the following proposition: 
\begin{proposition} \label{proposition:D4-deviants} We have the following $2$-dissections:
{\allowdisplaybreaks \begin{align}
D(0,4)&= 2-2q^2g(-q^2;q^{16})-2\cdot \frac{\overline{J}_{4,8}\overline{J}_{6,16}}{J_4} 
 + \frac{1}{2}\cdot\frac{\overline{J}_{1,2}\overline{J}_{1,4}}{J_4} 
 +\frac{1}{4}\cdot \frac{J_{1,2}J_{1,4}}{J_4},\label{equation:D04-deviant} \\
D(1,4)=D(3,4)&= -1+q^2g(-q^2;q^{16})+q^5g(-q^6;q^{16})+\frac{\overline{J}_{4,8}J_{1,4}}{J_4}
-\frac{1}{4}\cdot\frac{J_{1,2}J_{1,4}}{J_4},\label{equation:D14-deviant}\\
D(2,4)&=-2q^5g(-q^6;q^{16})+2\cdot q  \frac{\overline{J}_{4,8}\overline{J}_{2,16}}{J_4} 
  -\frac{1}{2}\cdot\frac{\overline{J}_{1,2}\overline{J}_{1,4}}{J_4}
  +\frac{1}{4}\cdot \frac{J_{1,2}J_{1,4}}{J_4}.\label{equation:D24-deviant}
\end{align}}%
\end{proposition}

\begin{proof}[Proof of Theorem \ref{theorem:D4-deviants}]
Using identity (\ref{equation:jsplit}) and collecting terms gives
\begin{align}
\frac{J_{1,2}J_{1,4}}{J_4}
&=\frac{1}{J_4}\cdot \Big [ \overline{J}_{4,8}-q \overline{J}_{0,8}\Big ] \cdot \Big [  \overline{J}_{6,16}-q \overline{J}_{14,16}\Big ] \notag\\
&=\frac{1}{J_4}\cdot \Big [ \overline{J}_{4,8}\overline{J}_{6,16}+q^2\overline{J}_{0,8}\overline{J}_{14,16}\Big ]
-\frac{q}{J_4}\cdot \Big [ \overline{J}_{4,8}\overline{J}_{14,16}+\overline{J}_{0,8}\overline{J}_{6,16}\Big ],\label{equation:D4-rewrite-1}
\end{align}
as well as
\begin{equation}
\frac{\overline{J}_{1,2}\overline{J}_{1,4}}{J_4}
=\frac{1}{J_4}\cdot \Big [ \overline{J}_{4,8}\overline{J}_{6,16}+q^2\overline{J}_{0,8}\overline{J}_{14,16}\Big ]
+\frac{q}{J_4}\cdot \Big [ \overline{J}_{4,8}\overline{J}_{14,16}+\overline{J}_{0,8}\overline{J}_{6,16}\Big ],\label{equation:D4-rewrite-2}
\end{equation}
and
\begin{equation}
\frac{\overline{J}_{4,8}J_{1,4}}{J_4}
=\frac{1}{J_4}\cdot \Big [ \overline{J}_{4,8}\overline{J}_{6,16}-q\cdot \overline{J}_{4,8}\overline{J}_{14,16}\Big ].\label{equation:D4-rewrite-3}
\end{equation}
Rewrite identities  $(\ref{equation:D04-deviant})$--$(\ref{equation:D24-deviant})$ using $(\ref{equation:D4-rewrite-1})$--$(\ref{equation:D4-rewrite-3})$ and collect terms.
\end{proof}

\begin{proof}[Proof of Proposition \ref{proposition:D4-deviants}]
Using (\ref{equation:rankdeviant-def}), we have
{\allowdisplaybreaks \begin{align}
D(0,4)&=\frac{1}{4}\sum_{j=0}^3\Big [(1-i^j)(1+i^jg(i^j;q))\Big ]\notag\\
&=\frac{1}{4}\sum_{j=0}^3\Big [1-i^{j}+i^{j}g(i^{j};q)-i^{2j}g(i^{j};q)\Big ]\notag\\
&=1+\frac{1}{4}\Big [ \big [ g(1;q)-g(-1;q)\big ]-\big [ g(1;q)+g(-1;q)\big ]\label{equation:D04-g-form}\\
&\ \ \ \ \  +i\big [ g(i;q)-g(-i;q)\big ]+\big [ g(i;q)+g(-i;q)\big ] \Big ].\notag
\end{align}}%
Corollary \ref{corollary:rootsof1} gives
{\allowdisplaybreaks \begin{align}
D(0,4)&=1+\frac{1}{4}\Big [ -4+4qg(-q;q^4)-4qg(q;q^4)\notag\\
&\ \ \ \ \ +\lim_{x\rightarrow 1}2J_2\overline{J}_{1,4}^2\cdot \Big ( \frac{1}{xj(-q^3x^2;q^4)j(x^2;q^2)}-\frac{1}{j(-qx^2;q^4)j(x^2;q^2)}\Big )\notag\\
&\ \ \ \ \ +2J_2\overline{J}_{1,4}^2\cdot \Big (\frac{i}{ij(q^3;q^4)j(-1;q^2)}+ \frac{1}{j(q;q^4)j(-1,q^2)}\Big )\Big ]\notag\\
&=qg(-q;q^4)-qg(q;q^4)
+ \frac{J_2\overline{J}_{1,4}^2}{J_{1,4}\overline{J}_{0,2}}\notag\\
&\ \ \ \ \ +\lim_{x\rightarrow 1}\frac{J_2\overline{J}_{1,4}^2}{2}\cdot  \frac{j(-qx^2;q^4)-xj(-q^3x^2;q^4)}{xj(-qx^2;q^4)j(-q^3x^2;q^4)j(x^2;q^2)} \notag\\ 
&=qg(-q;q^4)-qg(q;q^4)+\frac{1}{2}\cdot \frac{J_2^5}{J_1^2J_4^2}\cdot  \frac{J_2^2}{J_{1}}\cdot \frac{1}{J_4}\notag\\
&\ \ \ \ \ +\lim_{x\rightarrow 1}\frac{J_2\overline{J}_{1,4}^2}{2}\cdot  \frac{j(x;q)}{xj(-qx^2;q^4)j(-q^3x^2;q^4)j(x^2;q^2)} \notag\\ 
&=qg(-q;q^4)-qg(q;q^4)+ \frac{1}{2}\cdot \frac{\overline{J}_{1,2}\overline{J}_{1,4}}{J_4} +\frac{1}{4}\cdot \frac{J_{1,2}J_{1,4}}{J_4}, \label{equation:D04-prefinal}
\end{align}}%
where we have used (\ref{equation:jsplit}) and (\ref{equation:1.12}).  Employing Corollary \ref{corollary:rootsof1} again, we obtain
\begin{equation}
D(0,4)=2-2q^2g(-q^2;q^{16})-2\cdot \frac{\overline{J}_{4,8}\overline{J}_{6,16}}{J_4} 
 + \frac{1}{2}\cdot \frac{\overline{J}_{1,2}\overline{J}_{1,4}}{J_4} +\frac{1}{4}\cdot \frac{J_{1,2}J_{1,4}}{J_4}. \label{equation:D04-final}
\end{equation}

Using (\ref{equation:rankdeviant-def}), we have
{\allowdisplaybreaks \begin{align}
D(1,4)&=-1+\frac{1}{4}\Big [\big [g(1;q)+g(-1;q)\big ]-\big[ g(1;q)-g(-1;q)\big ]\label{equation:D14-g-form}\\
&\ \ \ \ \ \ \ \ \ \ \ \ \ \ \ +\big [g(i;q)+g(-i;q)\big ]-i\big [g(i;q)-g(-i;q)\big]\Big ].\notag
\end{align}}%
Using Corollary \ref{corollary:rootsof1}, we obtain
{\allowdisplaybreaks \begin{align}
D(1,4)&=-1+\frac{1}{4}\Big [4-4qg(-q;q^4)\notag\\
&\ \ \ \ \ +\lim_{x\rightarrow 1}2J_2\overline{J}_{1,4}^2\cdot \Big ( \frac{1}{j(-qx^2;q^4)j(x^2;q^2)}-\frac{1}{xj(-q^3x^2;q^4)j(x^2;q^2)}\Big )\notag\\
&\ \ \ \ \ +2J_2\overline{J}_{1,4}^2\cdot \Big ( \frac{1}{j(q;q^4)j(-1,q^2)}-\frac{i}{ij(q^3;q^4)j(-1;q^2)}\Big )\Big ]\notag\\
&=-qg(-q;q^4)
+\lim_{x\rightarrow 1}\frac{J_2\overline{J}_{1,4}^2}{2}\cdot  \frac{xj(-q^3x^2;q^4)-j(-qx^2;q^4)}{xj(-qx^2;q^4)j(-q^3x^2;q^4)j(x^2;q^2)}\notag\\
&=-qg(-q;q^4)
-\lim_{x\rightarrow 1}\frac{J_2\overline{J}_{1,4}^2}{2}\cdot  \frac{j(x;q)}{xj(-qx^2;q^4)j(-q^3x^2;q^4)j(x^2;q^2)}\notag\\
&=-qg(-q;q^4)-\frac{1}{4}\cdot \frac{J_{1,2}J_{1,4}}{J_4},\label{equation:D14-prefinal}
\end{align}}%
where we have again used (\ref{equation:jsplit}) and (\ref{equation:1.12}).  Proposition \ref{proposition:gsplit} then yields
\begin{equation}
D(1,4)=-1+q^2g(-q^2;q^{16})+q^5g(-q^6;q^{16})+\frac{\overline{J}_{4,8}J_{1,4}}{J_4}
-\frac{1}{4}\cdot \frac{J_{1,2}J_{1,4}}{J_4}.\label{equation:D14-final}
\end{equation}

Using (\ref{equation:rankdeviant-def}) gives
{\allowdisplaybreaks \begin{align}
D(2,4)&=\frac{1}{4}\Big [ \big [g(1;q)-g(-1;q)\big ]-\big [ g(1;q)+g(-1;q)\big ]\label{equation:D24-g-form}\\
&\ \ \ \ \  -i\big [ g(i;q)-g(-i;q)\big ]-\big [ g(i;q)+g(-i;q) \big ] \Big ].\notag
\end{align}}%
Using Corollary \ref{corollary:rootsof1}, we obtain
{\allowdisplaybreaks \begin{align}
D(2,4)
&=\frac{1}{4}\Big [ 4qg(-q;q^4)+4qg(q;q^4) -4\frac{J_2\overline{J}_{1,4}^2}{J_{1,4}\overline{J}_{0,2}}\notag \\
&\ \ \ \ \  +\lim_{x\rightarrow 1}2J_2\overline{J}_{1,4}^2
\cdot \Big ( \frac{1}{xj(-q^3x^2;q^4)j(x^2;q^2)}-\frac{1}{j(-qx^2;q^4)j(x^2;q^2)}\Big ) \Big ]\notag\\
&=qg(-q;q^4) +qg(q;q^4)\notag \\
&\ \ \ \ \  -\frac{1}{2}\cdot  \frac{\overline{J}_{1,2}\overline{J}_{1,4}}{J_4}
 +\lim_{x\rightarrow 1}\frac{J_2\overline{J}_{1,4}^2}{2}\cdot  \frac{j(x;q)}{xj(-qx^2;q^4)j(-q^3x^2;q^4)j(x^2;q^2)}\notag\\
 &=qg(-q;q^4) +qg(q;q^4) -\frac{1}{2}\cdot  \frac{\overline{J}_{1,2}\overline{J}_{1,4}}{J_4}
 +\frac{1}{4}\cdot \frac{J_{1,2}J_{1,4}}{J_4},\label{equation:D24-prefinal}
 \end{align}}%
where we have used (\ref{equation:jsplit}) and (\ref{equation:1.12}).  Employing Corollary \ref{corollary:rootsof1} again yields
 \begin{equation}
 D(2,4)= -2q^5g(-q^6;q^{16})+2q\cdot \frac{\overline{J}_{4,8}\overline{J}_{2,16}}{J_4} 
   -\frac{1}{2}\cdot  \frac{\overline{J}_{1,2}J_{2,4}}{J_1}+\frac{1}{4}\cdot \frac{J_{1,2}J_{1,4}}{J_4}, \label{equation:D24-final}
\end{equation}
which completes the proof.
\end{proof}

 \section{Proof of Theorem \ref{theorem:ABCKM-Thm1.6}}\label{section:ABCKM-Thm1.6}

Recalling $(\ref{equation:D04-deviant-alt})$ and $(\ref{equation:D24-deviant-alt})$ and regrouping terms, we have
{\allowdisplaybreaks \begin{align}
D&(0,4)-D(2,4)\notag\\
&= 2-2q^2g(-q^2;q^{16}) +2q^5g(-q^6;q^{16})\notag \\
&\ \ \ \ \  - \frac{\overline{J}_{4,8}\overline{J}_{6,16}}{J_4}
+q^2\cdot \frac{\overline{J}_{0,8}\overline{J}_{14,16}}{J_4}
-q \cdot \frac{\overline{J}_{4,8}\overline{J}_{14,16}}{J_4}
+q \cdot \frac{\overline{J}_{0,8}\overline{J}_{6,16}}{J_4}\notag\\
&= 2-2q^2g(-q^2;q^{16}) +2q^5g(-q^6;q^{16})
 -\frac{1}{J_4}\cdot \Big [  \overline{J}_{4,8}-q\overline{J}_{0,8}\Big ] \cdot\Big [\overline{J}_{6,16}+q\overline{J}_{14,16}\Big ]\notag \\
&= 2-2q^2g(-q^2;q^{16}) +2q^5g(-q^6;q^{16})
-\frac{J_{1,2}\overline{J}_{1,4}}{J_4},\label{equation:ABCKM-Thm1.6-prefinal}
\end{align}}%
where we have used (\ref{equation:jsplit}).   Elementary product rearrangements give
\begin{equation}
\frac{\overline{J}_{1,4}J_{1,2}}{J_4}=\frac{J_{2,4}}{J_4}\cdot J_{1,4}
=\frac{J_{2,4}}{J_4} \cdot \Big [ \overline{J}_{6,16}-q\overline{J}_{14,16}\Big ],\label{equation:ABCKM-productid}
\end{equation}
where we have again used (\ref{equation:jsplit}).  Rewriting (\ref{equation:ABCKM-Thm1.6-prefinal}) with (\ref{equation:ABCKM-productid}) gives Theorem \ref{theorem:ABCKM-Thm1.6}.

\section{On rank deviations modulo $8$}\label{section:D8-prelim}

\begin{theorem} \label{theorem:D8-deviants} We have the following $2$-dissections:
\begin{align}
D(0,8)&=\vartheta_8(-9,7,-3,5,-1,-1,5,-3)+G_8(1,-1,0,0), \label{equation:D08-deviant-final}\\
D(1,8)=D(7,8)&=\vartheta_8(7,-5,-3,1,3,-1,-3,1)+\tfrac{1}{2}\cdot G_8(-1,1,1,1), \label{equation:D18-deviant-final}\\
D(2,8)=D(6,8)&=\vartheta_8(-1,-1,5,-3,-1,-1,5,-3)+ G_8(0,0,0,-1), \label{equation:D28-deviant-final}\\
D(3,8)=D(5,8)&=\vartheta_8(-1,3,-3,1,-5,7,-3,1)+ \tfrac{1}{2}\cdot G_8(1,1,-1,1), \label{equation:D38-deviant-final}\\
D(4,8)&=\vartheta_8(-1,-1,5,-3,7,-9,-3,5)+ G_8(-1,-1,0,0), \label{equation:D48-deviant-final}
\end{align}
where
\begin{align}
\vartheta_8&(a_0,a_1,a_2,a_3,a_4,a_5,a_6,a_7)\\
&:=\frac{1}{8J_4}\Big [ a_0\cdot \overline{J}_{4,8}\overline{J}_{28,64}
+a_1\cdot q^4  \overline{J}_{0,8}\overline{J}_{52,64}
+a_2\cdot q \overline{J}_{4,8}\overline{J}_{20,64}
+a_3\cdot q \overline{J}_{0,8}\overline{J}_{28,64}\notag \\
&\ \ \ \ \ +a_4\cdot q^2  \overline{J}_{0,8}\overline{J}_{20,64}
+a_5\cdot q^6 \overline{J}_{4,8}\overline{J}_{60,64}
+a_6\cdot q^3 \overline{J}_{4,8}\overline{J}_{52,64}
+a_7\cdot q^7 \overline{J}_{0,8}\overline{J}_{60,64}\Big ]  ,\notag
\end{align}
and
\begin{align}
G_8(b_0,b_1,b_2,b_3)&:=b_0\cdot (1+q^{2}g(q^{2};q^{16 }))+b_1\cdot (-1+q^{2}g(-q^{2};q^{16}))\\
&\ \ \ \ \ +b_2\cdot q^{5}g(q^{6};q^{16})+b_3 \cdot q^{5}g(-q^{6};q^{16}).\notag
\end{align}
\end{theorem}

Theorem \ref{theorem:D8-deviants} is an immediate consequence of the following proposition:
\begin{proposition} \label{proposition:D8-deviants} We have the following $2$-dissections:
{\allowdisplaybreaks \begin{align}
D(0,8)&=2+q^2g(q^2;q^{16}) -q^2g(-q^2;q^{16})\label{equation:D08-deviant-pre} \\
&\ \ \ \ \ -\frac{\overline{J}_{4,8}\overline{J}_{6,16}}{J_4}
-\frac{\overline{J}_{4,8}J_{6,16}}{J_4}
+\frac{1}{4}\frac{\overline{J}_{1,2}\overline{J}_{1,4}}{J_4} 
+\frac{1}{8}\frac{J_{1,2}J_{1,4}}{J_4}
+\frac{1}{2}\frac{J_8\overline{J}_{1,2}}{J_{4}^2J_{16}}\cdot J_{1,8}\overline{J}_{3,8}, \notag \\
D(1,8)&=-1-\frac{1}{2}q^2g(q^2;q^{16})+\frac{1}{2}q^2g(-q^2;q^{16})
+\frac{1}{2}q^5g(q^6;q^{16}) +\frac{1}{2}q^5g(-q^{6};q^{16})\label{equation:D18-deviant-pre} \\
&\ \ \ \ \   +\frac{1}{2}\frac{\overline{J}_{4,8}J_{1,4}}{J_4}
-\frac{1}{2}q\frac{\overline{J}_{4,8}J_{2,16}}{J_4}
+\frac{1}{2}\frac{\overline{J}_{4,8}J_{6,16}}{J_4}-\frac{1}{8}\frac{J_{1,2}J_{1,4}}{J_4} 
+\frac{1}{2}q\frac{\overline{J}_{1,2}J_{14,16}}{J_4},\notag \\
D(2,8)&=-q^5g(-q^6;q^{16})+q\frac{\overline{J}_{4,8}\overline{J}_{14,16}}{J_4}
-\frac{1}{4}\frac{\overline{J}_{1,2}\overline{J}_{1,4}}{J_4} 
+\frac{1}{8}\frac{J_{1,2}J_{1,4}}{J_4},\label{equation:D28-deviant-pre} \\
D(3,8)&= \frac{1}{2}q^2g(q^2;q^{16})+\frac{1}{2}q^2g(-q^2;q^{16})
-\frac{1}{2}q^5g(q^6;q^{16}) +\frac{1}{2}q^5g(-q^{6};q^{16}) \label{equation:D38-deviant-pre}\\
&\ \ \ \ \  +\frac{1}{2}\frac{\overline{J}_{4,8}J_{1,4}}{J_4}
+\frac{1}{2}q\frac{\overline{J}_{4,8}J_{2,16}}{J_4}
-\frac{1}{2}\frac{\overline{J}_{4,8}J_{6,16}}{J_{4}}
-\frac{1}{8}\frac{J_{1,2}J_{1,4}}{J_4} 
-\frac{1}{2}q \frac{\overline{J}_{1,2}J_{2,16}}{J_4},  \notag \\
D(4,8)&=-q^2g(q^2;q^{16}) -q^2g(-q^2;q^{16})\label{equation:D48-deviant-pre} \\
&\ \ \ \ \ -\frac{\overline{J}_{4,8}\overline{J}_{6,16}}{J_4}
+\frac{\overline{J}_{4,8}J_{6,16}}{J_4}
+\frac{1}{4}\frac{\overline{J}_{1,2}\overline{J}_{1,4}}{J_4} 
+\frac{1}{8}\frac{J_{1,2}J_{1,4}}{J_4}
-\frac{1}{2}\frac{J_8\overline{J}_{1,2}}{J_{4}^2J_{16}}\cdot J_{1,8}\overline{J}_{3,8}. \notag
\end{align}}%
\end{proposition}

\begin{proof}[Proof of Theorem \ref{theorem:D8-deviants}]  Using $(\ref{equation:jsplit})$ and $(\ref{equation:H1Thm1.1})$ gives
{\allowdisplaybreaks \begin{align}
 \frac{J_8}{J_{4}^2J_{16}}\cdot \overline{J}_{1,2}\cdot J_{1,8}\overline{J}_{3,8}
&= \frac{J_8}{J_4^2J_{16}}
\cdot \Big [ \overline{J}_{4,8}+q\overline{J}_{0,8}\Big ] \cdot  \Big [ J_{6,16}J_{12,16}-qJ_{14,16}J_{4,16}\Big ] \label{equation:D8-rewrite-1} \\
&=\frac{1}{J_4}\cdot \Big [ \overline{J}_{4,8}J_{6,16}-q^2\overline{J}_{0,8}J_{14,16}\Big ] 
 + q\cdot \frac{1}{J_4}\cdot \Big [ \overline{J}_{0,8}J_{6,16}-\overline{J}_{4,8}J_{14,16}\Big ]. \notag 
\end{align}}%
Four more consequences of $(\ref{equation:jsplit})$ read
\begin{equation}
\begin{split}
\overline{J}_{6,16}=\overline{J}_{28,64}+q^6\overline{J}_{60,64}, \ 
\overline{J}_{2,16}=\overline{J}_{20,64}+q^2\overline{J}_{52,64}, \\
J_{6,16}=\overline{J}_{28,64}-q^6\overline{J}_{60,64}, \ 
J_{2,16}=\overline{J}_{20,64}-q^2\overline{J}_{52,64}. \label{equation:base-splits}
\end{split}
\end{equation}
Rewrite Proposition \ref{proposition:D8-deviants} using $(\ref{equation:jsplit})$ with $(\ref{equation:D4-rewrite-1})$--$(\ref{equation:D4-rewrite-3})$, $(\ref{equation:D8-rewrite-1})$, $(\ref{equation:base-splits})$  and collect terms.
\end{proof}

\begin{proof}[Proof of Proposition \ref{proposition:D8-deviants}]
The proofs for the five identities are all similar, so we will only do the first two.  Using (\ref{equation:rankdeviant-def}), we have
{\allowdisplaybreaks \begin{align}
D(0,8)&=1+\frac{1}{8}\Big [  -\big [ g(1;q)+g(-1;q)\big ]+\big [g(1;q)-g(-1;q)\big ] \label{equation:D08-g-form}\\
&\ \ \ \ \ \ \ \ \ \ +\big [ g(i;q)+g(-i;q)\big ]+i\big [g(i;q)-g(-i;q) \big ]\notag\\
&\ \ \ \ \ \ \ \ \ \ -i\big [ g(\zeta_8;q)+g(-\zeta_8;q)\big ]+\zeta_8\big [ g(\zeta_8;q)-g(-\zeta_8;q)\big ]\notag\\
&\ \ \ \ \ \ \ \ \ \ +i[g(\zeta_8^{-1};q) +g(-\zeta_8^{-1};q) \big ] +\zeta_8^{-1}\big [ g(\zeta_8^{-1};q) -g(-\zeta_8^{-1};q)\big ]\Big ].\notag
\end{align}}%
Similarly, we have
{\allowdisplaybreaks \begin{align}
D(1,8)&=-1+\frac{1}{8}\Big [ \big [g(1;q) +g(-1;q)\big ] -\big [ g(1;q)-g(-1;q)\big ] \label{equation:D18-g-form} \\
&\ \ \ \ \ \ \ \ \ \ +\big [g(i;q)+g(-i;q) \big ] -i\big [g(i;q)-g(-i;q) \big ] \notag \\
&\ \ \ \ \ \ \ \ \ \ +\big [ g(\zeta_8;q)+g(-\zeta_8;q)\big ] -\zeta_8\big [ g(\zeta_8;q)-g(-\zeta_8;q)\big ] \notag \\
&\ \ \ \ \ \ \ \ \ \ +\big [g(\zeta_8^{-1};q)+ g(-\zeta_8^{-1};q)\big ] -\zeta_8^{-1} \big [g(\zeta_8^{-1};q)-g(-\zeta_8^{-1};q) \big ] \Big ].\notag
\end{align}}%

Applying Corollary \ref{corollary:rootsof1} to (\ref{equation:D08-g-form}) and combining terms produces
{\allowdisplaybreaks \begin{align*}
D(0,8)&=1+\frac{1}{8}\Big [-2+ 4qg(-q;q^4)\\
&\ \ \ \ \ \ \ \ \ \ +\lim_{x\rightarrow 1}2J_2\overline{J}_{1,4}^2\Big [\frac{1}{xj(-q^3x^2;q^4)j(x^2;q^2)}-\frac{1}{j(-qx^2;q^4)j(x^2;q^2)} \Big ] \\
&\ \ \ \ \ +\big [ -2qg(q;q^4)+\frac{2J_2\overline{J}_{1,4}^2}{J_{1,4}\overline{J}_{0,2}}\big ] 
+i\big [ 2i+2iqg(q;q^4)+\frac{2J_2\overline{J}_{1,4}^2}{iJ_{3,4}\overline{J}_{0,2}}\big ] \\
&\ \ \ \ \ -i\big [ -2qg(-iq;q^4)+\frac{2J_2\overline{J}_{1,4}^2}{j(-iq;q^4)j(i;q^2)}\big ] \\
&\ \ \ \ \ \ \ \ \ \ +\zeta_8\big [ -2\zeta_8^{-1}-2\zeta_8qg(iq;q^4)+\zeta_8^{-1}\frac{2J_2\overline{J}_{1,4}^2}{j(-iq^3;q^4)j(i;q^2)}\big ] \\
&\ \ \ \ \ +i\big [ -2qg(iq;q^4)+\frac{2J_2\overline{J}_{1,4}^2}{j(iq;q^4)j(-i;q^2)}\big ] \\
&\ \ \ \ \ \ \ \ \ \ +\zeta_8^{-1}\big [ -2\zeta_8-2\zeta_8^{-1}qg(-iq;q^4)+\zeta_8\frac{2J_2\overline{J}_{1,4}^2}{j(iq^3;q^4)j(-i;q^2)}\big ] \Big ] \\
&=1+\frac{1}{8}\Big [-8 -4qg(q;q^4) + 4qg(-q;q^4)\\
&\ \ \ \ \ +\frac{4J_2\overline{J}_{1,4}^2}{J_{1,4}\overline{J}_{0,2}} 
+\lim_{x\rightarrow 1}2J_2\overline{J}_{1,4}^2
\Big [\frac{j(-qx^2;q^4)-xj(-q^3x^2;q^4)}{xj(-q^3x^2;q^4)j(-qx^2;q^4)j(x^2;q^2)} \Big ] \\ 
&\ \ \ \ \  -4iqg(iq;q^4)+4iqg(-iq;q^4)-i\frac{2J_2\overline{J}_{1,4}^2}{j(-iq;q^4)j(i;q^2)}
+\frac{2J_2\overline{J}_{1,4}^2}{j(-iq^3;q^4)j(i;q^2)} \\
&\ \ \ \ \  +i\frac{2J_2\overline{J}_{1,4}^2}{j(iq;q^4)j(-i;q^2)} 
+\frac{2J_2\overline{J}_{1,4}^2}{j(iq^3;q^4)j(-i;q^2)} \Big ]. 
\end{align*}}%
We rewrite the first quotient and use (\ref{equation:jsplit}) and (\ref{equation:1.12}) to evaluate the expression inside the limit.  We combine pairwise the last four quotients using (\ref{equation:1.8}) and (\ref{equation:1.7}).  This gives
{\allowdisplaybreaks \begin{align*}
D(0,8)
&=\frac{1}{8}\Big [ -4qg(q;q^4) + 4qg(-q;q^4)
  +\frac{2\overline{J}_{1,2}\overline{J}_{1,4}}{J_4} +\frac{J_{1,2}J_{1,4}}{J_4}\\ 
&\ \ \ \ \  -4iqg(iq;q^4)+4iqg(-iq;q^4)+\frac{4J_2\overline{J}_{1,4}^2}{j(-iq;q^4)j(-i;q^2)}
+\frac{4J_2\overline{J}_{1,4}^2}{j(iq;q^4)j(i;q^2)} \Big ]. 
\end{align*}}%
Regrouping terms and combining the last two quotients gives
{\allowdisplaybreaks \begin{align*}
D(0,8)&=\frac{1}{8}\Big [ -4q\big [g(q;q^4) -g(-q;q^4)\big]  -4iq\big [ g(iq;q^4)-g(-iq;q^4)\big ]
 +\frac{2\overline{J}_{1,2}\overline{J}_{1,4}}{J_4}\\
 &\ \ \ \ \  +\frac{J_{1,2}J_{1,4}}{J_4}
   +\frac{4J_2\overline{J}_{1,4}^2}{\overline{J}_{0,8}\overline{J}_{0,4}} \cdot \frac{J_8J_4}{J_4^2J_2^2}
\cdot \big [j(-iq;q^4)j(-i;q^2)+j(iq;q^4)j(i;q^2) \big ] \Big ]. 
\end{align*}}%
Using (\ref{equation:1.10}) and noting that $j(iq^2;q^4)=j(-iq^2;q^4)=J_8^2/J_{16}$, we have
{\allowdisplaybreaks \begin{align*}
D(0,8)
&=\frac{1}{8}\Big [ -4q\big [g(q;q^4) -g(-q;q^4)\big]  -4iq\big [ g(iq;q^4)-g(-iq;q^4)\big ] \\
&\ \ \ \ \ +\frac{2\overline{J}_{1,2}\overline{J}_{1,4}}{J_4} +\frac{J_{1,2}J_{1,4}}{J_4}\\
&\ \ \ \ \  +\frac{4J_2\overline{J}_{1,4}^2}{\overline{J}_{2,8}\overline{J}_{0,4}} \cdot \frac{J_8J_4}{J_4^2J_2^2}
\cdot \frac{J_2}{J_4^2}\cdot \frac{J_8^2}{J_{16}}\cdot \big [j(-iq;q^4)j(-i;q^4)+j(iq;q^4)j(i;q^4) \big ] \Big ] \\
&=\frac{1}{8}\Big [ -4q\big [g(q;q^4) -g(-q;q^4)\big]  -4iq\big [ g(iq;q^4)-g(-iq;q^4)\big ] \\
&\ \ \ \ \ +\frac{2\overline{J}_{1,2}\overline{J}_{1,4}}{J_4} +\frac{J_{1,2}J_{1,4}}{J_4}
   +\frac{4J_2\overline{J}_{1,4}^2}{\overline{J}_{2,8}\overline{J}_{0,4}} \cdot \frac{J_8J_4}{J_4^2J_2^2}
\cdot \frac{J_2}{J_4^2}\cdot \frac{J_8^2}{J_{16}}\cdot2J_{1,8}\overline{J}_{3,8}\Big ], 
\end{align*}}%
where we have used (\ref{equation:H1Thm1.2B}).   Rewriting the last term we have
{\allowdisplaybreaks \begin{align*}
D(0,8)
&=\frac{1}{8}\Big [ -4q\big [g(q;q^4) -g(-q;q^4)\big]  -4iq\big [ g(iq;q^4)-g(-iq;q^4)\big ] \\
&\ \ \ \ \ +\frac{2\overline{J}_{1,2}\overline{J}_{1,4}}{J_4} +\frac{J_{1,2}J_{1,4}}{J_4}
+ \frac{4J_8\overline{J}_{1,2}}{J_{4}^2J_{16}}\cdot J_{1,8}\overline{J}_{3,8}\Big ]. 
\end{align*}}%
Using Corollary \ref{corollary:rootsof1}, rewriting the new theta quotients and collecting terms, we have
{\allowdisplaybreaks \begin{align*}
D(0,8)
&=\frac{1}{8}\Big [ -4q\big [-2q^{-1}+2qg(-q^2;q^{16})+\frac{2J_{8}\overline{J}_{4,16}^2}{q\overline{J}_{14,16}J_{2,8}} \big]\\
&\ \ \ \ \   -4iq\big [ 2iq^{-1}+2iqg(q^2;q^{16})+\frac{2J_{8}\overline{J}_{4,16}^2}{iqJ_{14,16}\overline{J}_{2,8}}\big ] \\
&\ \ \ \ \ +\frac{2\overline{J}_{1,2}\overline{J}_{1,4}}{J_4} +\frac{J_{1,2}J_{1,4}}{J_4}
+\frac{4J_8\overline{J}_{1,2}}{J_{4}^2J_{16}}\cdot J_{1,8}\overline{J}_{3,8}\Big ] \\
&=2+q^2g(q^2;q^{16}) -q^2g(-q^2;q^{16})\\
&\ \ \ \ \ -\frac{\overline{J}_{4,8}\overline{J}_{6,16}}{J_4}
-\frac{\overline{J}_{4,8}J_{6,16}}{J_4}
+\frac{1}{4}\frac{\overline{J}_{1,2}\overline{J}_{1,4}}{J_4}
+\frac{1}{8}\frac{J_{1,2}J_{1,4}}{J_4}
+\frac{1}{2} \frac{J_8\overline{J}_{1,2}}{J_{4}^2J_{16}}\cdot J_{1,8}\overline{J}_{3,8}, 
\end{align*}}%
where we have rewritten products and combined terms.

Applying Corollary \ref{corollary:rootsof1} to (\ref{equation:D18-g-form}) and combining terms produces
{\allowdisplaybreaks \begin{align*}
D(1,8)&=-1+\frac{1}{8}\Big [2- 4qg(-q;q^4)\\
&\ \ \ \ \ \ \ \ \ \ -\lim_{x\rightarrow 1}2J_2\overline{J}_{1,4}^2\Big [\frac{1}{xj(-q^3x^2;q^4)j(x^2;q^2)}-\frac{1}{j(-qx^2;q^4)j(x^2;q^2)} \Big ] \\
&\ \ \ \ \ +\big [ -2qg(q;q^4)+\frac{2J_2\overline{J}_{1,4}^2}{J_{1,4}\overline{J}_{0,2}}\big ] 
-i\big [ 2i+2iqg(q;q^4)+\frac{2J_2\overline{J}_{1,4}^2}{iJ_{3,4}\overline{J}_{0,2}}\big ] \\
&\ \ \ \ \ +\big [ -2qg(-iq;q^4)+\frac{2J_2\overline{J}_{1,4}^2}{j(-iq;q^4)j(i;q^2)}\big ] \\
&\ \ \ \ \ \ \ \ \ \ -\zeta_8\big [ -2\zeta_8^{-1}-2\zeta_8qg(iq;q^4)+\zeta_8^{-1}\frac{2J_2\overline{J}_{1,4}^2}{j(-iq^3;q^4)j(i;q^2)}\big ] \\
&\ \ \ \ \ +\big [ -2qg(iq;q^4)+\frac{2J_2\overline{J}_{1,4}^2}{j(iq;q^4)j(-i;q^2)}\big ] \\
&\ \ \ \ \ \ \ \ \ \ -\zeta_8^{-1}\big [ -2\zeta_8-2\zeta_8^{-1}qg(-iq;q^4)+\zeta_8\frac{2J_2\overline{J}_{1,4}^2}{j(iq^3;q^4)j(-i;q^2)}\big ] \Big ] \\
&=-1+\frac{1}{8}\Big [8-4qg(-q;q^4) \\
&\ \ \ \ \ -2q\big [ g(iq;q^4)+g(-iq;q^4)\big ]+2iq\big [ g(iq;q^4)-g(-iq;q^4)\big] \\
&\ \ \ \ \  -\frac{J_{1,2}J_{1,4}}{J_4} 
 +\frac{2J_2\overline{J}_{1,4}^2}{j(-iq;q^4)j(i;q^2)} -\frac{2J_2\overline{J}_{1,4}^2}{j(-iq^3;q^4)j(i;q^2)} \\
&\ \ \ \ \  +\frac{2J_2\overline{J}_{1,4}^2}{j(iq;q^4)j(-i;q^2)} -\frac{2J_2\overline{J}_{1,4}^2}{j(iq^3;q^4)j(-i;q^2)} \Big ], 
\end{align*}}%
where we have used (\ref{equation:jsplit}) and (\ref{equation:1.12}) to evaluate the limit.  Next we have
{\allowdisplaybreaks \begin{align*}
D(1,8)&=\frac{1}{8}\Big [-4qg(-q;q^4) 
 -2q\big [ g(iq;q^4)+g(-iq;q^4)\big ]+2iq\big [ g(iq;q^4)-g(-iq;q^4)\big] \\
&\ \ \ \ \  -\frac{J_{1,2}J_{1,4}}{J_4} 
 +(1+i)\frac{2J_2\overline{J}_{1,4}^2}{j(-iq;q^4)j(i;q^2)} -(1+i)\frac{2J_2\overline{J}_{1,4}^2}{j(iq;q^4)j(i;q^2)} \Big ], 
\end{align*}}%
where we have combined theta quotients using (\ref{equation:1.8}) and (\ref{equation:1.7}).  Combining fractions and using (\ref{equation:jsplit}) gives
{\allowdisplaybreaks \begin{align*}
D(1,8)
&=\frac{1}{8}\Big [-4qg(-q;q^4)
 -2q\big [ g(iq;q^4)+g(-iq;q^4)\big ]+2iq\big [ g(iq;q^4)-g(-iq;q^4)\big] \\
&\ \ \ \ \  -\frac{J_{1,2}J_{1,4}}{J_4} 
+(1+i)\frac{2J_2\overline{J}_{1,4}^2}{j(-iq;q^4)j(iq;q^4)j(i;q^2)} \big [ j(iq;q^4)-j(-iq;q^4)\big ]\Big ] \\
&=\frac{1}{8}\Big [-4q g(-q;q^4)
 -2q\big [ g(iq;q^4)+g(-iq;q^4)\big ]  -\frac{J_{1,2}J_{1,4}}{J_4} \\
 &\ \ \ \ \ +2iq\big [ g(iq;q^4)-g(-iq;q^4)\big] 
+\frac{1+i}{1-i}\frac{2J_2\overline{J}_{1,4}^2}{\overline{J}_{2,8}} \frac{J_8}{J_4^2} \frac{J_4}{J_2J_8} 
\big [  -2iqJ_{14,16}\big ] \Big ] \\
&=\frac{1}{8}\Big [-4q g(-q;q^4)  -2q\big [ g(iq;q^4)+g(-iq;q^4)\big ]\\
&\ \ \ \ \ +2iq\big [ g(iq;q^4)-g(-iq;q^4)\big] 
 -\frac{J_{1,2}J_{1,4}}{J_4} 
+4q\frac{\overline{J}_{1,2}J_{14,16}}{J_4}\Big ], 
\end{align*}}%
where we have simplified the last quotient.  Proposition \ref{proposition:gsplit} and Corollary \ref{corollary:rootsof1} yield
{\allowdisplaybreaks \begin{align*}
D(1,8)
&=\frac{1}{8}\Big [-4q\big [   q^{-1}-qg(-q^2;q^{16})-q^4g(-q^{6};q^{16})-\frac{J_{8}J_{8,16}^2}{q\overline{J}_{1,4}\overline{J}_{6,8}}   \big ] \\
&\ \ \ \ \  -2q\big [ -2q^4g(q^6;q^{16})+\frac{2J_8\overline{J}_{4,16}^2}{J_{6,16}\overline{J}_{2,8}}\big ]\\
&\ \ \ \ \ +2iq\big [ 2iq^{-1}+2iqg(q^2;q^{16})+\frac{2J_8\overline{J}_{4,16}^2}{iqJ_{14,16}\overline{J}_{2,8}}\big] 
-\frac{J_{1,2}J_{1,4}}{J_4} 
+4q\frac{\overline{J}_{1,2}J_{14,16}}{J_4}\Big ] \\
&=-1-\frac{1}{2}q^2g(q^2;q^{16})+\frac{1}{2}q^2g(-q^2;q^{16})
+\frac{1}{2}q^5g(q^6;q^{16}) +\frac{1}{2}q^5g(-q^{6};q^{16})\\
&\ \ \ \ \   +\frac{1}{2} \frac{\overline{J}_{4,8}J_{1,4}}{J_4}
 -\frac{1}{2} q \frac{\overline{J}_{4,8}J_{2,16}}{J_4}
+\frac{1}{2} \frac{\overline{J}_{4,8}J_{6,16}}{J_4}
-\frac{1}{8}\frac{J_{1,2}J_{1,4}}{J_4} 
+\frac{1}{2} q \frac{\overline{J}_{1,2}J_{14,16}}{J_4},
\end{align*}}%
where we have rewritten products and collected terms.
\end{proof}

\section{Proof of Theorem \ref{theorem:ABCKM-Thm1.7}}\label{section:ABCKM-Thm1.7}

Recalling $(\ref{equation:D08-deviant-final})$ and $(\ref{equation:D48-deviant-final})$, we have 
{\allowdisplaybreaks \begin{align*}
D&(0,8)-D(4,8)\\
&=2+2q^2g(q^2;q^{16})
-\frac{\overline{J}_{4,8}\overline{J}_{28,64}}{J_4}
+q^4  \frac{\overline{J}_{0,8}\overline{J}_{52,64}}{J_4}
-q \frac{\overline{J}_{4,8}\overline{J}_{20,64}}{J_4}
+q \frac{\overline{J}_{0,8}\overline{J}_{28,64}}{J_4} \\
&\ \ \ \ \ - q^2  \frac{\overline{J}_{0,8}\overline{J}_{20,64}}{J_4}
+q^6\frac{\overline{J}_{4,8}\overline{J}_{60,64}}{J_4}
+q^3 \frac{\overline{J}_{4,8}\overline{J}_{52,64}}{J_4}
-q^7 \frac{\overline{J}_{0,8}\overline{J}_{60,64}}{J_4}\\
&=2+2q^2g(q^2;q^{16})
-\frac{\overline{J}_{4,8}(\overline{J}_{28,64}-q^6\overline{J}_{60,64})}{J_4} 
-q^2  \frac{\overline{J}_{0,8}(\overline{J}_{20,64}-q^2\overline{J}_{52,64})}{J_4} \\
&\ \ \ \ \  -q \frac{\overline{J}_{4,8}(\overline{J}_{20,64}-q^2\overline{J}_{52,64})}{J_4} 
+q \frac{\overline{J}_{0,8}(\overline{J}_{28,64}-q^6\overline{J}_{60,64})}{J_4} \\
&=2+2q^2g(q^2;q^{16})
 -\frac{ \overline{J}_{4,8}J_{6,16}}{J_4}
-q^2 \frac{\overline{J}_{0,8}J_{14,16}}{J_4} 
-q \frac{\overline{J}_{4,8}J_{14,16}}{J_4} 
+q \frac{ \overline{J}_{0,8}J_{6,16}}{J_4}\\
&=2+2q^2g(q^2;q^{16})
-\frac{1}{J_{4}} \big [  \overline{J}_{4,8}J_{6,16}
+q^2 \overline{J}_{0,8}J_{14,16}\big ] 
+q\frac{1}{J_4} \big [   \overline{J}_{0,8}J_{6,16}
-\overline{J}_{4,8}J_{14,16}\big ],
\end{align*}}%
where we have regrouped terms, used $(\ref{equation:jsplit})$, and regrouped terms again.  Using $(\ref{equation:hecke-phi-id})$  with $x\mapsto-1$, $q\mapsto -q^2$  and $(\ref{equation:hecke-phi-id-alt})$  with $x\mapsto 1$, $q\mapsto -q^2$  yields 
{\allowdisplaybreaks \begin{align}
D&(0,8)-D(4,8)\notag \\
&=2+2q^2g(q^2;q^{16})-\frac{1}{J_4} \frac{j(-q^2;-q^6)}{J_8} \overline{J}_{6,8}J_{6,16}
+q\cdot \frac{1}{J_4} \frac{j(-q^2;-q^6)}{J_8} \overline{J}_{6,8}J_{14,16}\notag\\
&=2+2q^2g(q^2;q^{16})-\frac{1}{J_4} \frac{j(-q^2;q^{12})j(q^{8};q^{12})}{J_{6,24}J_8} \frac{J_4^2}{J_2}
\Big [ J_{6,16}-qJ_{14,16}\Big ]\notag\\
&=2+2q^2g(q^2;q^{16})-\frac{\overline{J}_{2,4}J_{6,16}}{J_4}  
+q \frac{\overline{J}_{2,4}J_{2,16}}{J_4}  ,\notag
\end{align}}%
where we have used  (\ref{equation:1.11}) and then simplified the product.

Recalling $(\ref{equation:D18-deviant-final})$ and $(\ref{equation:D38-deviant-final})$, we have 
{\allowdisplaybreaks \begin{align*}
D(1,8)-D(3,8)
&=-1- q^2g(q^2;q^{16})+q^5g(q^6;q^{16})  \\
&\ \ \ \ \   + \frac{\overline{J}_{4,8}\overline{J}_{28,64}}{J_4}
-q^4 \frac{\overline{J}_{0,8}\overline{J}_{52,64}}{J_4} 
+q^2 \frac{\overline{J}_{0,8}\overline{J}_{20,64}}{J_4} 
-q^6 \frac{\overline{J}_{4,8}\overline{J}_{60,64}}{J_4}\\
&=-1- q^2g(q^2;q^{16})+q^5g(q^6;q^{16})  \\
&\ \ \ \ \   + \frac{\overline{J}_{4,8}(\overline{J}_{28,64}-q^6\overline{J}_{60,64})}{J_4}
+q^2 \frac{\overline{J}_{0,8}(\overline{J}_{20,64}-q^2\overline{J}_{52,64})}{J_4}\\
&=-1-q^2g(q^2;q^{16})+q^5g(q^6;q^{16})
+\frac{\overline{J}_{4,8}J_{6,16}}{J_4}
+q^2  \frac{\overline{J}_{0,8}J_{14,16}}{J_4}\\
&=-1-q^2g(q^2;q^{16})+q^5g(q^6;q^{16})
+\frac{1}{J_4} \big [ \overline{J}_{4,8}J_{6,16}
+q^2 \overline{J}_{0,8}J_{14,16}\big ],
\end{align*}}%
where we have regrouped terms, used $(\ref{equation:jsplit})$, and regrouped terms again.  Using $(\ref{equation:hecke-phi-id})$  with $x\mapsto-1$, $q\mapsto -q^2$  and simplifying with $(\ref{equation:1.11})$ gives
{\allowdisplaybreaks \begin{align*}
D(1,8)-D(3,8)
&=-1+q^5g(q^6;q^{16})-q^2g(q^2;q^{16})
+\frac{1}{J_4} \frac{j(-q^2;-q^6)}{J_8} \overline{J}_{6,8}J_{6,16}\\
&=-1+q^5g(q^6;q^{16})-q^2g(q^2;q^{16})
+\frac{\overline{J}_{2,4}J_{6,16}}{J_4}.
\end{align*}}%

\section{Proof of Theorem \ref{theorem:DC4-deviants}}\label{section:DC4-deviants}
The three identities have similar proofs, so we only do the first.  Using $(\ref{equation:crankdeviant-def})$, we have
{\allowdisplaybreaks \begin{align}
D_C(0,4)
&=\frac{1}{4}\sum_{j=1}^{3}i^{-0\cdot j}\frac{(q)_{\infty}}{(i^jq)_{\infty}(i^{-j}q)_{\infty}}
=\frac{1}{4} \Big [ \frac{(q)_{\infty}}{(-q^2;q^2)_{\infty}}
+\frac{(q)_{\infty}}{(-q;q)_{\infty}^2}+\frac{(q)_{\infty}}{(-q^2;q^2)_{\infty}}\Big ]\notag \\
&=\frac{1}{4}\Big [2 \frac{J_1J_2}{J_4}+\frac{J_1^3}{J_2^2}\Big ] 
=\frac{1}{2}\frac{J_{1,2}\overline{J}_{1,4}}{J_4}+\frac{1}{4}  \frac{J_{1,2}J_{1,4}}{J_4}.\label{equation:DC04-prefinal}
\end{align}}%
Using $(\ref{equation:jsplit})$ on both theta functions in the numerator and expanding gives
\begin{equation}
\frac{J_{1,2}\overline{J}_{1,4}}{J_4}
=\frac{1}{J_4}\cdot \Big [ \overline{J}_{4,8}\overline{J}_{6,16}-q^2\overline{J}_{0,8}\overline{J}_{14,16}\Big ]
+\frac{q}{J_4}\cdot \Big [ \overline{J}_{4,8}\overline{J}_{14,16}-\overline{J}_{0,8}\overline{J}_{6,16}\Big ].\label{equation:DC-rewrite-1}
\end{equation}
Rewritting $(\ref{equation:DC04-prefinal})$ with $(\ref{equation:DC-rewrite-1})$ and $(\ref{equation:D4-rewrite-1})$  and collecting terms produces $(\ref{equation:DC04-deviant-final})$.

\section{On crank deviations modulo $8$}\label{section:crank-prelims}
We have analogous dissections for crank deviants modulo $8$.

\begin{theorem} \label{theorem:DC8-deviants} We have the following $4$-dissections:
\begin{align}
D_C(0,8)&=\vartheta_8(3,-1,1,-3,-1,3,1,-3)+\vartheta_8^{\prime}(1,-1,-1,1),\label{equation:DC08-deviant-final}\\
D_C(1,8)=D_C(7,8)&=\vartheta_8(-1,-1,1,1,-1,-1,1,1)+\vartheta_8^{\prime}(0,1,0,-1),\label{equation:DC18-deviant-final}\\
D_C(2,8)=D_C(6,8)&=\vartheta_8(-1,3,-3,1,3,-1,-3,1),\label{equation:DC28-deviant-final}\\
D_C(3,8)=D_C(5,8)&=\vartheta_8(-1,-1,1,1,-1,-1,1,1)+\vartheta_8^{\prime}(0,-1,0,1),\label{equation:DC38-deviant-final}\\
D_C(4,8)&=\vartheta_8(3,-1,1,-3,-1,3,1,-3)+\vartheta_8^{\prime}(-1,1,1,-1),\label{equation:DC48-deviant-final}
\end{align}
where
\begin{align}
\vartheta_8&(a_0,a_1,a_2,a_3,a_4,a_5,a_6,a_7)\\
&:=\frac{1}{8J_4}\Big [ a_0\cdot \overline{J}_{4,8}\overline{J}_{28,64}
+a_1\cdot q^4  \overline{J}_{0,8}\overline{J}_{52,64}
+a_2\cdot q \overline{J}_{4,8}\overline{J}_{20,64}
+a_3\cdot q \overline{J}_{0,8}\overline{J}_{28,64}\notag \\
&\ \ \ \ \ +a_4\cdot q^2  \overline{J}_{0,8}\overline{J}_{20,64}
+a_5\cdot q^6 \overline{J}_{4,8}\overline{J}_{60,64}
+a_6\cdot q^3 \overline{J}_{4,8}\overline{J}_{52,64}
+a_7\cdot q^7 \overline{J}_{0,8}\overline{J}_{60,64}\Big ]  ,\notag
\end{align}
and
\begin{equation}
\vartheta_8^{\prime}(c_0,c_1,c_2,c_3):=\frac{1}{2J_4}\Big [ c_0\cdot J_{4,8}\overline{J}_{28,64}
+c_1\cdot q J_{4,8}\overline{J}_{20,64}
+c_2\cdot q^6 J_{4,8}\overline{J}_{60,64}
+c_3\cdot  q^3 J_{4,8}\overline{J}_{52,64}\Big ] .
\end{equation}
\end{theorem}

\begin{proof}[Proof of Theorem \ref{theorem:DC8-deviants}]
The proofs for each of the five identities are all similar, so we prove only the first and third identities.    Using $(\ref{equation:crankdeviant-def})$ and collecting like terms, we have
{\allowdisplaybreaks \begin{align*}
D_C(0,8)
&=\frac{1}{8}\Big [ 2\frac{(q)_{\infty}}{(\zeta_8q;q)_{\infty}(\zeta_8^{-1}q;q)_{\infty}}
+2\frac{(q)_{\infty}}{(-\zeta_8q;q)_{\infty}(-\zeta_8^{-1}q;q)_{\infty}}\\
&\ \ \ \ \ +2\frac{(q)_{\infty}}{(iq;q)_{\infty}(-iq;q)_{\infty}}
+\frac{(q)_{\infty}}{(-q;q)_{\infty}(-q;q)_{\infty}}\Big ] \\
&=\frac{1}{8}\Big [ 2\frac{(1-\zeta_8)J_1^2}{j(\zeta_8;q)}
+2\frac{(1+\zeta_8)J_1^2}{j(-\zeta_8;q)}
+2\frac{J_{1,2}\overline{J}_{1,4}}{J_4}+\frac{J_{1,2}J_{1,4}}{J_4}\Big ],
\end{align*}}%
where we have rewritten the four terms using our theta function notation.
{\allowdisplaybreaks \begin{align*}
D_C(0,8)
&=\frac{1}{8}\Big [ 2\frac{J_1^2}{j(\zeta_8;q)j(-\zeta_8;q)}
\Big ( (1-\zeta_8){j(-\zeta_8;q)}+(1+\zeta_8)j(\zeta_8;q)\Big )\\
&\ \ \ \ \ +2\frac{J_{1,2}\overline{J}_{1,4}}{J_4}+\frac{J_{1,2}J_{1,4}}{J_4}\Big ] \\
&=\frac{1}{8}\Big [ 2\frac{J_1^2}{j(i;q^2)}\frac{J_2}{J_1^2}
\Big ( (1-\zeta_8)\Big ( j(-qi;q^4) +\zeta_8 j(qi;q^4)\Big )\\
&\ \ \ \ \  +(1+\zeta_8)\Big ( j(-qi;q^4) -\zeta_8 j(qi;q^4)\Big )\Big )
+2\frac{J_{1,2}\overline{J}_{1,4}}{J_4}+\frac{J_{1,2}J_{1,4}}{J_4}\Big ] \\
&=\frac{1}{8}\Big [ \frac{4}{1-i}\frac{J_1^2J_4}{J_2J_8}\frac{J_2}{J_1^2}
\Big (  j(-qi;q^4) - ij(qi;q^4)  \Big )
 +2\frac{J_{1,2}\overline{J}_{1,4}}{J_4}+\frac{J_{1,2}J_{1,4}}{J_4}\Big ], 
\end{align*}}%
where we have used identities $(\ref{equation:1.12})$ and $(\ref{equation:jsplit})$, distributed the products, and used the fact that $j(i;q^2)=(1-i)J_2J_8/J_4$.  Using $(\ref{equation:jsplit})$ again
\begin{align}
D_C(0,8)
&=\frac{1}{8}\Big [ \frac{4}{1-i}\frac{J_1^2J_4}{J_2J_8}\frac{J_2}{J_1^2}
\Big (  J_{6,16}+qiJ_{14,16} - i\Big (J_{6,16}-qiJ_{14,16} \Big ) \Big )
+2\frac{J_{1,2}\overline{J}_{1,4}}{J_4}+\frac{J_{1,2}J_{1,4}}{J_4}\Big ]\notag \\
&=\frac{1}{8}\Big [ \frac{4}{1-i}\frac{J_1^2J_4}{J_2J_8}\frac{J_2}{J_1^2}
\Big (  (1-i)J_{6,16}-(1-i)qJ_{14,16}   \Big )
+2\frac{J_{1,2}\overline{J}_{1,4}}{J_4}+\frac{J_{1,2}J_{1,4}}{J_4}\Big ] \notag \\
&=\frac{1}{2}\frac{J_{4,8}J_{6,16}}{J_4}
-q\frac{1}{2}\frac{J_{4,8}J_{2,16}}{J_4}
+\frac{1}{4}\frac{J_{1,2}\overline{J}_{1,4}}{J_4}+\frac{1}{8}\frac{J_{1,2}J_{1,4}}{J_4}.\label{equation:DC08-pre-prefinal}
\end{align}
Rewritting $(\ref{equation:DC08-pre-prefinal})$ with $(\ref{equation:DC-rewrite-1})$ and $(\ref{equation:D4-rewrite-1})$  and collecting terms produces $(\ref{equation:DC04-deviant-final})$ yields

{\allowdisplaybreaks \begin{align}
D_C(0,8)
&=\frac{3}{8} \frac{\overline{J}_{4,8}\overline{J}_{6,16}}{J_4}
-\frac{1}{8}q^2\frac{\overline{J}_{0,8}\overline{J}_{14,16}}{J_4}
+\frac{1}{8}q \frac{\overline{J}_{4,8}\overline{J}_{14,16}}{J_4}
-\frac{3}{8}q\frac{\overline{J}_{0,8}\overline{J}_{6,16}}{J_4}\notag \\
&\ \ \ \ \ +\frac{1}{2}\frac{J_{4,8}J_{6,16}}{J_4}
-q\frac{1}{2}\frac{J_{4,8}J_{2,16}}{J_4}.\label{equation:DC08-prefinal}
\end{align}}%
Rewriting $(\ref{equation:DC08-prefinal})$ using $(\ref{equation:base-splits})$ and collecting terms finally results in our $4$-dissection $(\ref{equation:DC08-deviant-final})$.

Using $(\ref{equation:crankdeviant-def})$ and noting pairwise cancellation, 
{\allowdisplaybreaks \begin{align}
D_C(2,8)
&=-\frac{1}{4}\frac{(q)_{\infty}}{(iq;q)_{\infty}(-iq;q)_{\infty}}
+\frac{1}{8}\frac{(q)_{\infty}}{(-q;q)_{\infty}^2}
=-\frac{1}{4}\frac{(q)_{\infty}}{(-q^2;q^2)_{\infty}}
+\frac{1}{8}\frac{(q)_{\infty}}{(-q;q)_{\infty}^2}\notag \\
&=-\frac{1}{4}\frac{J_1J_2}{J_4}
+\frac{1}{8}\frac{J_1^3}{J_2^2}
=-\frac{1}{4}\frac{J_{1,2}\overline{J}_{1,4}}{J_4}
+\frac{1}{8}\frac{J_{1,2}J_{1,4}}{J_4}.\label{equation:DC28-pre-final}
\end{align}}%
Rewriting $(\ref{equation:DC28-pre-final})$ with theta-dissections  $(\ref{equation:DC-rewrite-1})$ and $(\ref{equation:D4-rewrite-1})$ and collecting terms produces
\begin{equation}
D_C(2,8)
=-\frac{1}{8} \frac{\overline{J}_{4,8}\overline{J}_{6,16}}{J_4}
+\frac{3}{8}q^2\frac{\overline{J}_{0,8}\overline{J}_{14,16}}{J_4}
-\frac{3}{8}q \frac{\overline{J}_{4,8}\overline{J}_{14,16}}{J_4}
+\frac{1}{8}q\frac{\overline{J}_{0,8}\overline{J}_{6,16}}{J_4}.\label{equation:DC28-2-dissection}
\end{equation}
Rewriting $(\ref{equation:DC28-2-dissection})$ with $(\ref{equation:base-splits})$, we arrive our $4$-dissection  $(\ref{equation:DC28-deviant-final})$.
\end{proof}

\section{On Rank-Crank Identities of Lewis and Santa-Gadea}\label{section:rank-crank}

Theorems \ref{theorem:D8-deviants} and \ref{theorem:DC8-deviants} prove the relations $(\ref{equation:NC-10})$--$(\ref{equation:NC-17})$.  We give some examples.

\subsection{Identities $(\ref{equation:NC-14})$--$(\ref{equation:NC-17})$}
To prove $(\ref{equation:NC-14})$ and $(\ref{equation:NC-15})$, we do not need to compute the entire $4$-dissection for $D(3,8)$.  We only need to determine which terms contribute to $q$-powers $q^n$ where $n\equiv 0,1 \pmod 4$.  We see that the first line of $(\ref{equation:D38-deviant-final})$ does not contribute.  Using Corollary \ref{corollary:rootsof1}, we note that the two expressions
{\allowdisplaybreaks \begin{align}
q^2g(q^2;q^{16})+q^2g(-q^2;q^{16}) &=-2q^{18}g(-q^{20};q^{64})
+2q^2\frac{J_{32}\overline{J}_{16,64}^2}{\overline{J}_{20,64}J_{4,32}},\label{equation:g2-plus}\\
q^5g(q^6;q^{16})-q^5g(-q^6;q^{16}) &= -2q^{-1}+2q^{3}g(-q^{4};q^{64})
+2q^{-1}\frac{J_{32}\overline{J}_{16,64}^2}{\overline{J}_{60,64}J_{12,32}}\label{equation:g6-minus}
\end{align}}%
are supported on $q$-powers $q^n$ where $n\equiv 2,3\pmod 4$ respectively.  Hence contributions can only come from the last two lines in $(\ref{equation:D38-deviant-final})$.    Comparing $q$-powers $q^n$ where $n\equiv 0 \pmod 4$ in $(\ref{equation:DC28-deviant-final})$ and $(\ref{equation:D38-deviant-final})$ proves $(\ref{equation:NC-14})$.  Similiarly, comparing $q$-powers $q^n$ where $n\equiv 1 \pmod 4$ in $(\ref{equation:DC28-deviant-final})$ and $(\ref{equation:D38-deviant-final})$ proves $(\ref{equation:NC-15})$.

For $(\ref{equation:NC-16})$ and $(\ref{equation:NC-17})$, we proceed analogously.  We first  note that the top line of $D(1,8)$ in $(\ref{equation:D18-deviant-final})$ does not contribute.  Using Corollary \ref{corollary:rootsof1}, we see that
\begin{align}
q^2g(q^2;q^{16})-q^2g(-q^2;q^{16}) &= -2+2q^{12}g(-q^{12};q^{64})
+2\frac{J_{32}\overline{J}_{16,64}^2}{\overline{J}_{52,64}J_{4,32}}, \label{equation:g2-minus}\\
q^5g(q^6;q^{16})+q^5g(-q^6;q^{16}) &= -2q^{21}g(-q^{28};q^{64})
+2q^{5}\frac{J_{32}\overline{J}_{16,64}^2}{\overline{J}_{28,64}J_{12,32}}\label{equation:g6-plus}
\end{align}
are supported on $q$-powers $q^n$ where $n\equiv 0,1\pmod 4$ respectively.  Any potential contribuation can only come from the last two lines in $(\ref{equation:D18-deviant-final})$.  Comparing $q$-powers $q^n$ where $n\equiv 2 \pmod 4$ in $(\ref{equation:DC28-deviant-final})$ and $(\ref{equation:D18-deviant-final})$ proves $(\ref{equation:NC-16})$.  Likewise, comparing $q$-powers $q^n$ where $n\equiv 3 \pmod 4$ in $(\ref{equation:DC28-deviant-final})$ and $(\ref{equation:D18-deviant-final})$ proves $(\ref{equation:NC-17})$.

\subsection{Identity $(\ref{equation:NC-11})$}
Recalling $(\ref{equation:DC08-deviant-final})$, $(\ref{equation:DC18-deviant-final})$, $(\ref{equation:DC38-deviant-final})$ and $(\ref{equation:DC48-deviant-final})$, we have
{\allowdisplaybreaks \begin{align}
D_C&(0,8)+D_C(1,8)\notag \\
&=\frac{1}{4}\cdot  \frac{\overline{J}_{4,8}\overline{J}_{28,64}}{J_4}
-\frac{1}{4}\cdot q^4 \frac{\overline{J}_{0,8}\overline{J}_{52,64}}{J_4}
+\frac{1}{4}\cdot q  \frac{\overline{J}_{4,8}\overline{J}_{20,64}}{J_4}
-\frac{1}{4}\cdot q \frac{\overline{J}_{0,8}\overline{J}_{28,64}}{J_4}\label{equation:DC08-DC18-sum}\\
&\ \ \ \ \ -\frac{1}{4}\cdot q^2\frac{\overline{J}_{0,8}\overline{J}_{20,64}}{J_4}
+\frac{1}{4}\cdot  q^6 \frac{\overline{J}_{4,8}\overline{J}_{60,64}}{J_4}
+\frac{1}{4}\cdot q^3  \frac{\overline{J}_{4,8}\overline{J}_{52,64}}{J_4}
-\frac{1}{4}\cdot q^7 \frac{\overline{J}_{0,8}\overline{J}_{60,64}}{J_4}\notag \\
&\ \ \ \ \ +\frac{1}{2}\cdot \frac{J_{4,8}\overline{J}_{28,64}}{J_4}
-\frac{1}{2}\cdot q^6 \frac{J_{4,8}\overline{J}_{60,64}}{J_4},\notag \\
D_C&(3,8)+D_C(4,8)\notag \\
&= \frac{1}{4} \cdot \frac{ \overline{J}_{4,8}\overline{J}_{28,64}}{J_4}
-\frac{1}{4} \cdot q^4 \frac{\overline{J}_{0,8}\overline{J}_{52,64}}{J_4}
 +\frac{1}{4}\cdot q \frac{ \overline{J}_{4,8}\overline{J}_{20,64}}{J_4}
-\frac{1}{4}\cdot q \frac{\overline{J}_{0,8}\overline{J}_{28,64}}{J_4}\label{equation:DC38-DC48-sum}\\
&\ \ \ \ \ -\frac{1}{4}\cdot q^2 \frac{\overline{J}_{0,8}\overline{J}_{20,64}}{J_4}
+ \frac{1}{4} \cdot q^6  \frac{ \overline{J}_{4,8}\overline{J}_{60,64}}{J_4}
+\frac{1}{4} \cdot q^3 \frac{ \overline{J}_{4,8}\overline{J}_{52,64}}{J_4}
-\frac{1}{4}\cdot q^7  \frac{\overline{J}_{0,8}\overline{J}_{60,64}}{J_4}\notag \\
&\ \ \ \ \ -\frac{1}{2}\cdot  \frac{J_{4,8}\overline{J}_{28,64}}{J_4}
+\frac{1}{2}\cdot q^6 \frac{J_{4,8}\overline{J}_{60,64}}{J_4}.\notag
\end{align}}%
Comparing $q$-powers $q^n$ where $n\equiv 1 \pmod 4$ in $(\ref{equation:DC08-DC18-sum})$ and $(\ref{equation:DC38-DC48-sum})$ proves the first equality in $(\ref{equation:NC-11})$.   

Noting $(\ref{equation:D18-deviant-final})$ and $(\ref{equation:D28-deviant-final})$, we have
{\allowdisplaybreaks \begin{align}
D&(1,8)+D(2,8)\notag \\
&=-1-\frac{1}{2} q^2g(q^2;q^{16})+\frac{1}{2} q^2g(-q^2;q^{16})
+\frac{1}{2}q^5g(q^6;q^{16}) -\frac{1}{2} q^5g(-q^{6};q^{16}) \label{equation:D18-D28-sum}\\
&\ \ \ \ \   +\frac{3}{4}\cdot \frac{\overline{J}_{4,8}\overline{J}_{28,64}}{J_4}
-\frac{3}{4}\cdot q^4 \frac{\overline{J}_{0,8}\overline{J}_{52,64}}{J_4}
+\frac{1}{4}\cdot q \frac{\overline{J}_{4,8}\overline{J}_{20,64}}{J_4}
-\frac{1}{4}\cdot q \frac{\overline{J}_{0,8}\overline{J}_{28,64}}{J_4}\notag \\
&\ \ \ \ \  +\frac{1}{4}\cdot q^2 \frac{\overline{J}_{0,8}\overline{J}_{20,64}}{J_4} 
-\frac{1}{4}\cdot q^6 \frac{\overline{J}_{4,8}\overline{J}_{60,64}}{J_4}
+\frac{1}{4}\cdot q^3 \frac{\overline{J}_{4,8}\overline{J}_{52,64}}{J_4}
-\frac{1}{4}\cdot q^7 \frac{\overline{J}_{0,8}\overline{J}_{60,64}}{J_4}. \notag 
\end{align}}%
By $(\ref{equation:g2-minus})$ and $(\ref{equation:g6-minus})$ we know that the first line of $(\ref{equation:D18-D28-sum})$ is supported on $q$-powers $q^n$ where $n\equiv 0,3\pmod 4$.  Hence we only need to consider the last two lines in $(\ref{equation:D18-D28-sum})$.  Comparing $q$-powers $q^n$ where $n\equiv 1 \pmod 4$ in $(\ref{equation:DC38-DC48-sum})$ and $(\ref{equation:D18-D28-sum})$ proves the second equality in $(\ref{equation:NC-11})$.

Noting $(\ref{equation:D38-deviant-final})$ and $(\ref{equation:D48-deviant-final})$, we have
{\allowdisplaybreaks \begin{align}
D&(3,8)+D(4,8)\notag\\
&=-\frac{1}{2} q^2g(q^2;q^{16})-\frac{1}{2} q^2g(-q^2;q^{16})
-\frac{1}{2}q^5g(q^6;q^{16}) +\frac{1}{2} q^5g(-q^{6};q^{16})\label{equation:D38-D48-sum} \\
&\ \ \ \ \   -\frac{1}{4}\cdot \frac{\overline{J}_{4,8}\overline{J}_{28,64}}{J_4}
+\frac{1}{4}\cdot q^4 \frac{\overline{J}_{0,8}\overline{J}_{52,64}}{J_4}
+\frac{1}{4}\cdot q \frac{\overline{J}_{4,8}\overline{J}_{20,64}}{J_4}
-\frac{1}{4}\cdot q \frac{\overline{J}_{0,8}\overline{J}_{28,64}}{J_4}\notag \\
&\ \ \ \ \ +\frac{1}{4}\cdot q^2 \frac{\overline{J}_{0,8}\overline{J}_{20,64}}{J_4}
-\frac{1}{4}\cdot q^6\frac{\overline{J}_{4,8}\overline{J}_{60,64}}{J_4}
-\frac{3}{4}\cdot q^3 \frac{\overline{J}_{4,8}\overline{J}_{52,64}}{J_4}
+\frac{3}{4}\cdot q^7 \frac{\overline{J}_{0,8}\overline{J}_{60,64}}{J_4}.\notag 
\end{align}}%
By $(\ref{equation:g2-plus})$ and $(\ref{equation:g6-minus})$ we know that the first line of $(\ref{equation:D38-D48-sum})$ is supported on $q$-powers $q^n$ where $n\equiv 2,3\pmod 4$.   Hence we only need to consider the last two lines in $(\ref{equation:D38-D48-sum})$.  Comparing $q$-powers $q^n$ where $n\equiv 1 \pmod 4$ in  $(\ref{equation:D18-D28-sum})$ and $(\ref{equation:D38-D48-sum})$ proves the final equality in $(\ref{equation:NC-11})$.

\section{On ranks and cranks with $M=5$ and  $7$}\label{section:rankscranks57}

The following two theorems give the dissections of the rank and crank deviations for $M=5$ and $7$.  The first half of each theorem is just a rewritten \cite[$(12)$--$(14)$]{HM2} and \cite[$(34)$--$(37)$]{HM2} respectively, where we used 
\begin{equation}
J_{1,5}J_{2,5}=J_1J_{5} \  \textup{and} \ J_{1,7}J_{2,7}J_{3,7}=J_1J_7^2. \label{equation:J5-J7}
\end{equation}
\begin{theorem}\label{theorem:M5-dissection} We have the following $5$-dissections:
\begin{align}
D(0,5)&=  2\cdot \vartheta_5( 2,2,-1,1)+2\cdot G_5(-1,0),\\
D(1,5)=D(4,5)&=  \vartheta_5(-1,-1,3,-3) + G_5(1,-1),\\
D(2,5)=D(3,5)&=\vartheta_5(-1,-1,-2,2) + G_5(0,1),
\end{align}
and
\begin{align}
D_C(0,5)&=2\cdot \vartheta_5(2,-3,-1,1),\label{equation:DC05-deviant}\\
D_C(1,5)=D_C(4,5)&= \vartheta_5( -1, 4, -2, -3),\label{equation:DC15-deviant}\\
D_C(2,5)=D_C(3,5)&= \vartheta_5( -1, -1, 3, 2),\label{equation:DC25-deviant}
\end{align}
where
\begin{equation}
\vartheta_5(a_0,a_1,a_2,a_3):=\frac{1}{5J_5^2}\Big [ a_0\cdot J_{10,25}^3+a_1\cdot q J_{5,25}J_{10,25}^2
+a_2\cdot q^2 J_{ 5,25}^2J_{10,25} +a_3 \cdot q^3 J_{5,25}^3\Big ] ,
\end{equation}
and
\begin{equation}
G_5(b_0,b_3):=b_0\cdot q^5g(q^5;q^{25})+b_3\cdot q^{8}g(q^{10};q^{25}).
\end{equation}
\end{theorem}

\begin{theorem}\label{theorem:M7-dissection} We have the following $7$-dissections:
\begin{align}
D(0,7)
&=2 \cdot \vartheta_7 (-4,3,-1,2,1,-2) + 2 \cdot G_7(1,0,0),\\
D(1,7)=D(6,7)
&= \vartheta_7 (6,-1,5,-3,2,3) + G_7(-1,1,0),\\
D(2,7)=D(5,7)
&=\vartheta_7 (-1,-1,-2,4,-5,3) + G_7(0,-1,1),\\
D(3,7)=D(4,7)
&= \vartheta_7 (-1,-1,-2,-3,2,-4) + G_7(0,0,-1),
\end{align}
and
\begin{align}
D_C(0,7)
&=2 \cdot \vartheta_7 (3,-4,-1,2,1,-2),\\
D_C(1,7)=D_C(6,7)
&= \vartheta_7 (-1,6,-2,-3,-5,3),\\
D_C(2,7)=D_C(5,7)
&= \vartheta_7 (-1,-1,5,-3,2,-4),\\
D_C(3,7)=D_C(4,7)
&= \vartheta_7 (-1,-1,-2,4,2,3),
\end{align}
where
\begin{align}
\vartheta_7(a_0,a_1,a_2,a_3,a_4,a_6)&:=\frac{1}{7J_7}\Big [ a_0\cdot J_{21,49}^2 + a_1 \cdot q J_{14,49}J_{21,49}+ a_2 \cdot q^2  J_{14,49}^2\label{equation:F7-def} \\
&\ \ \ \ \  + a_3 \cdot q^3  J_{7,49}J_{21,49}
+ a_4 \cdot q^4 J_{7,49}J_{14,49} + a_6 \cdot q^6  J_{7,49}^2\Big ] ,\notag
\end{align}
and
\begin{equation}
G_7(b_0,b_2,b_6):=b_0\cdot (1+q^{7}g(q^{7};q^{49}))+b_2\cdot q^{16}g(q^{21};q^{49})+b_6 \cdot q^{13}g(q^{14};q^{49}).
\end{equation}
\end{theorem}

\begin{proof} [Proof of Theorem \ref{theorem:M5-dissection}]  The proofs of identities $(\ref{equation:DC05-deviant})$--$(\ref{equation:DC25-deviant})$ are all similar, so we do only $(\ref{equation:DC15-deviant})$ as an example.  The $m=5$ specialization of $(\ref{equation:jsplitgen})$ with $(\ref{equation:1.8})$ and $(\ref{equation:1.7})$ gives
\begin{align}
j(\zeta_5;q)&=J_{10,25}-\zeta_5J_{15,25}+q\zeta_5^2J_{20,25}+q^6\zeta_5^4J_{30,25}\notag \\
&=(1-\zeta_5)J_{10,25}+q\zeta_5^2(1-\zeta_5^2)J_{5,25}.\label{equation:jroot-5}
\end{align}
Specializing the quintuple product identity $(\ref{equation:H1Thm1.0})$ with $q\mapsto q^{25}$ yields
\begin{equation}
j(q^{25}x^3;q^{75})+xj(q^{50}x^3;q^{75})=\frac{J_{25}j(x^2;q^{25})}{j(x;q^{25})}.\label{equation:quintuple-spec}
\end{equation}
Using $(\ref{equation:crankdeviant-def})$ and writing the summands over a common denominator, we have
{\allowdisplaybreaks \begin{align*}
D_C(1,5)&=\frac{1}{5}\Big [ \frac{(\zeta_5^{-1}+\zeta_5^{-4})(q)_{\infty}}{(\zeta_5q)_{\infty}(\zeta_5^{-1}q)_{\infty}}
+\frac{(\zeta_5^{-2}+\zeta_5^{-3})(q)_{\infty}}{(\zeta_5^2q)_{\infty}(\zeta_5^{-2}q)_{\infty}}\Big ] \\
&=\frac{J_1^2}{5}\Big [ \frac{(1-\zeta_5)(\zeta_5+\zeta_5^{4})}{j(\zeta_5;q)}
+\frac{(1-\zeta_5^2)(\zeta_5^2+\zeta_5^{3})}{j(\zeta_5^2;q)}\Big ] \\
&=\frac{J_1^2}{5}\Big [ \frac{(1-\zeta_5)(\zeta_5+\zeta_5^{4})j(\zeta_5^2;q)
+(1-\zeta_5^2)(\zeta_5^2+\zeta_5^{3})j(\zeta_5;q)}{j(\zeta_5;q)j(\zeta_5^2;q)}\Big ] \\
&=\frac{1}{5}\frac{J_1}{J_{5}}\Big [ \frac{(1-\zeta_5)(\zeta_5+\zeta_5^{4})j(\zeta_5^2;q)
+(1-\zeta_5^2)(\zeta_5^2+\zeta_5^{3})j(\zeta_5;q)}{(1-\zeta_5)(1-\zeta_5^2)}\Big ],
\end{align*}}%
where we have used product rearrangements to rewrite the denominator:
\begin{equation}
j(\zeta_5;q)j(\zeta_5^2;q)=(1-\zeta_5)(1-\zeta_5^2)J_1J_5.
\end{equation}
Using $(\ref{equation:jroot-5})$ and simplifying gives
{\allowdisplaybreaks \begin{align}
D_C(1,5)
&=\frac{1}{5}\frac{J_1}{J_{5}}\frac{1}{(1-\zeta_5)(1-\zeta_5^2)}
  \Big [(1-\zeta_5)(\zeta_5+\zeta_5^{4})\Big ((1-\zeta_5^2)J_{10,25}+q\zeta_5^4(1-\zeta_5^4)J_{5,25} \Big )\notag \\
&\ \ \ \ \ +(1-\zeta_5^2)(\zeta_5^2+\zeta_5^{3})\Big ((1-\zeta_5)J_{10,25}+q\zeta_5^2(1-\zeta_5^2)J_{5,25} \Big ) \Big ] \notag \\
&=\frac{1}{5}\frac{J_1}{J_{5}} \Big [-J_{10,25} + 3q J_{5,25} \Big ].\label{equation:DC15-prefinal}
\end{align}}%
The $m=5$ specialization of $(\ref{equation:jsplitgen})$ followed by $(\ref{equation:1.8})$, $(\ref{equation:1.7})$ and $(\ref{equation:quintuple-spec})$ yields
\begin{align}
j(q,q^3)&=J_{35,75}-qJ_{50,25}+q^5J_{65,75}-q^{12}J_{80,75}+q^{22}J_{95,75}\notag \\
&=J_{35,75}+q^5J_{65,75}-q^2(J_{20,75}+q^{10}J_{80,75})-qJ_{25}\notag \\
&=\frac{J_{25}J_{10,25}}{J_{5,25}}-q^2\frac{J_{25}J_{20,25}}{J_{10,25}}-qJ_{25}.\label{equation:J[1]-quintuple}
\end{align}
Note that $J_1=j(q;q^3)$ and substitute $(\ref{equation:J[1]-quintuple})$ into $(\ref{equation:DC15-prefinal})$:  
{\allowdisplaybreaks \begin{align}
D_C(1,5)
&=\frac{1}{5}\frac{J_{25}}{J_{5}}\Big [ \frac{J_{10,25}}{J_{5,25}}-q^2\frac{J_{20,25}}{J_{10,25}}-q\Big ] 
\Big [ -J_{10,25} +3qJ_{5,25}\Big ]\notag  \\
&=\frac{1}{5}\frac{J_{25}}{J_{5}}\Big [-\frac{J_{10,25}^2}{J_{5,25}}+4qJ_{10,25}
-2q^2J_{5,25}-3q^3\frac{J_{5,25}^2}{J_{10,25}}\Big ].
\end{align}}%
The result then follows from the product rearrangement $J_{5,25}J_{10,25}=J_5J_{25}$.
\end{proof}

\begin{proof} [Proof of Theorem \ref{theorem:M7-dissection}]  The proof is much the same as the previous one, so we provide only an outline.  Using $(\ref{equation:crankdeviant-def})$ and writing the summands over a common denominator
{\allowdisplaybreaks \begin{align}
D_C(a,7)&=\frac{1}{7}\frac{1}{J_{7}}\frac{1}{(1-\zeta_7)(1-\zeta_7^2)(1-\zeta_7^3)}
\Big [ (1-\zeta_7)(\zeta_7^{a}+\zeta_7^{-a})j(\zeta_7^2;q)j(\zeta_7^3;q)\label{equation:DCa7-gen}\\
&\ \ +(1-\zeta_7^2)(\zeta_7^{2a}+\zeta_7^{-2a})j(\zeta_7;q)j(\zeta_7^3;q)
 +(1-\zeta_7^3)(\zeta_7^{3a}+\zeta_7^{-3a})j(\zeta_7;q)j(\zeta_7^2;q)\Big ],\notag
\end{align}}%
where we have used product rearrangements to rewrite the denominator:
\begin{equation}
j(\zeta_7;q)j(\zeta_7^2;q)j(\zeta_7^3;q)=(1-\zeta_7)(1-\zeta_7^2)(1-\zeta_7^3)J_1^2J_7.
\end{equation}  
The $m=7$ specialization of $(\ref{equation:jsplitgen})$ with $(\ref{equation:1.8})$ and $(\ref{equation:1.7})$ gives
\begin{equation}
j(\zeta_7;q)=(1-\zeta_7)J_{21,49}-q\zeta_7^{-1}(1-\zeta_7^3)J_{14,49}
-q^3\zeta_7^3(1-\zeta_7^2)J_{7,49}.\label{equation:j7-split}
\end{equation}
Inserting $(\ref{equation:j7-split})$ along with analogs corresponding to $\zeta_7\rightarrow \zeta_7^2$ and $\zeta_7\rightarrow \zeta_7^3$ into $(\ref{equation:DCa7-gen})$ and collecting coefficients of the theta products of $(\ref{equation:F7-def})$ gives the desired results.
\end{proof}

\section{On Conjectures of Richard Lewis}
Lewis  \cite[Conjecture $1$]{L1} conjectured $4$-dissections for the one hundred combinations of
\begin{equation}
\sum_{n=0}^{\infty}\Big ( N(i,8;4n+k)-C(j,8;4n+k)\Big )q^n, \ \textup{where} \ 0\le i,j\le 4, \ 0\le k\le3.
\end{equation}
Frank Garvan has pointed out that Proposition \ref{proposition:gsplit} strengthens Theorem \ref{theorem:D8-deviants} to a $4$-dissection.  Using the new version of Theorem \ref{theorem:D8-deviants}, Theorem \ref{theorem:DC8-deviants}, and his thetaids package, he verified all one hundred conjectures \cite{G1}.   The truth of the one-hundred identities of \cite[Conjecture $1$]{L1} implies the veracity of the thirty-seven conjectured rank-crank inequalities \cite[Conjecture $2$]{L1} and the four evenness conjectures \cite[Conjecture $3$]{L1}.

We give an example to demonstrate how Theorems \ref{theorem:D8-deviants} and \ref{theorem:DC8-deviants} resolve Lewis's conjectures.  Identities $(0)$-$(3)$ of \cite[Conjecture $1$]{L1} are equivalent to the following:
\begin{proposition} \label{proposition:L-conj} We have the following $4$-dissection:
\begin{align}
D&(0,8)-D_C(0,8)\label{equation:L-conj-prop-id}\\
&=2q^{12}g(-q^{12};q^{64})-2q^8\frac{\overline{J}_{4,64}^2\overline{J}_{20,64}^2\overline{J}_{28,64}J_{64}^2}{J_{8,64}^2J_{16,64}J_{24,64}^2J_{32,64}}
+2q\cdot \frac{\overline{J}_{12,64}^2\overline{J}_{20,64}\overline{J}_{28,64}^2J_{64}^2}{J_{8,64}^2J_{16,64}J_{24,64}^2J_{32,64}}\notag \\
&\ \ \ \ \ -2q^2\cdot q^{8}\frac{\overline{J}_{4,64}^2\overline{J}_{12,64}\overline{J}_{20,64}\overline{J}_{28,64}J_{64}^2}{J_{8,64}^2J_{16,64}J_{24,64}^2J_{32,64}}
+2q^3\cdot q^4\frac{\overline{J}_{4,64}\overline{J}_{12,64}^2\overline{J}_{20,64}\overline{J}_{28,64}J_{64}^2}{J_{8,64}^2J_{16,64}J_{24,64}^2J_{32,64}}.\notag 
\end{align}
\end{proposition}
The truth of (\ref{equation:L-conj-prop-id}) implies three rank-crank inequalities $(1)$-$(3)$ of \cite[Conjecture $2$]{L1}:
\begin{proposition} \label{proposition:L-conj-ineq} We have the following inequalities
\begin{align}
N(0,8;4n+1)&>_2 C(0,8;4n+1),\label{equation:L-ineq-0}\\
N(0,8;4n+2)&<_2 C(0,8;4n+2),\label{equation:L-ineq-1}\\
N(0,8;4n+3)&>_1 C(0,8;4n+3),\label{equation:L-ineq-2}
\end{align} 
where $A_n>B_n$ means $A_n\ge B_n$ for all $n\in \mathbb{N}$ and $A_n>_m B_n$ means $A_n\ge B_n$ for all $n\ge m$.
\end{proposition}

\begin{proof} [Proof of Proposition \ref{proposition:L-conj}] We recall Theorems \ref{theorem:D8-deviants} and \ref{theorem:DC8-deviants} and identity (\ref{equation:g2-minus}).  Expanding $\overline{J}_{4,8}$, $\overline{J}_{0,8}$, and $J_{4,8}$ with (\ref{equation:jsplit}) gives
\begin{align}
D&(0,8)-D_C(0,8)\label{equation:L-conj-pre}\\
&=2q^{12}g(-q^{12};q^{64})+2\frac{J_{32}\overline{J}_{16,64}^2}{\overline{J}_{52,64}J_{4,32}}
- 2\frac{\overline{J}_{16,32}\overline{J}_{28,64}}{J_4}-q^4\frac{\overline{J}_{0,32}\overline{J}_{28,64}}{J_4}
+ 2 q^4 \frac{ \overline{J}_{8,32}\overline{J}_{52,64}}{J_4}\notag \\
&\ \ \ \ \ +2 q\frac{ \overline{J}_{8,32}\overline{J}_{28,64}}{J_4}
- q^5\frac{ \overline{J}_{0,32}\overline{J}_{20,64}}{J_4} 
- q^{10} \frac{\overline{J}_{0,32}\overline{J}_{60,64}}{J_4} 
+ q^{7} \frac{\overline{J}_{0,32}\overline{J}_{52,64}}{J_4}.\notag
\end{align}
Proving Proposition \ref{proposition:L-conj} is then reduced to proving the four identities:
\begin{subequations}
\begin{equation}
2\frac{J_{32}\overline{J}_{16,64}^2}{\overline{J}_{52,64}J_{4,32}}
- 2\frac{\overline{J}_{16,32}\overline{J}_{28,64}}{J_4}-q^4\frac{\overline{J}_{0,32}\overline{J}_{28,64}}{J_4}
+ 2 q^4 \frac{ \overline{J}_{8,32}\overline{J}_{52,64}}{J_4}
=-2q^8\frac{\overline{J}_{4,64}^2\overline{J}_{20,64}^2\overline{J}_{28,64}J_{64}^2}{J_{8,64}^2J_{16,64}J_{24,64}^2J_{32,64}},
\label{equation:lewis-000}
\end{equation}
\begin{equation}
2 \frac{\overline{J}_{8,32}\overline{J}_{28,64}}{J_4}
- q^4\frac{ \overline{J}_{0,32}\overline{J}_{20,64}}{J_4}
=  2\frac{\overline{J}_{12,64}^2\overline{J}_{20,64}\overline{J}_{28,64}^2J_{64}^2}{J_{8,64}^2J_{16,64}J_{24,64}^2J_{32,64}},
\label{equation:lewis-001}
\end{equation}
\begin{equation}
q^{8} \frac{\overline{J}_{0,32}\overline{J}_{60,64}}{J_4} 
= 2q^8\frac{\overline{J}_{4,64}^2\overline{J}_{12,64}\overline{J}_{20,64}\overline{J}_{28,64}J_{64}^2}{J_{8,64}^2J_{16,64}J_{24,64}^2J_{32,64}},
\label{equation:lewis-002}
\end{equation}
\begin{equation}
q^{4} \frac{\overline{J}_{0,32}\overline{J}_{52,64}}{J_4}
=2q^4\frac{\overline{J}_{4,64}\overline{J}_{12,64}^2\overline{J}_{20,64}\overline{J}_{28,64}J_{64}^2}{J_{8,64}^2J_{16,64}J_{24,64}^2J_{32,64}}.
\label{equation:lewis-003}
\end{equation}
\end{subequations}

We proceed in the order of difficulty.  Identities (\ref{equation:lewis-002}) and (\ref{equation:lewis-003}) are the easiest and are just elementary product rearrangements.

We prove identity (\ref{equation:lewis-001}).  Straightforward product rearrangements yield
\begin{equation}
J_{4}=J_{4,12}
=\frac{J_{8,64}^2J_{16,64}J_{24,64}^2J_{32,64}\overline{J}_{32,128}}
{\overline{J}_{4,64}\overline{J}_{12,64}\overline{J}_{20,64}\overline{J}_{28,64}J_{64}^2}.\label{equation:J4-expansion}
\end{equation}
Using (\ref{equation:J4-expansion}) and simplifying, Identity (\ref{equation:lewis-001}) is then seen to be equivalent to
\begin{equation}
2 \overline{J}_{8,32}\overline{J}_{28,64}
- q^4\overline{J}_{0,32}\overline{J}_{20,64}
=  2\frac{\overline{J}_{12,64}\overline{J}_{28,64}\overline{J}_{32,128}}{\overline{J}_{4,64}},
\end{equation}
which is equivalent to 
\begin{equation}
\overline{J}_{4,64} \overline{J}_{8,32}\overline{J}_{28,64}
- q^4\overline{J}_{4,64}\overline{J}_{20,64}\overline{J}_{32,128}
=  \overline{J}_{12,64}\overline{J}_{28,64}\overline{J}_{32,128}.\label{equation:lewis-001-prefinal}
\end{equation}
Identity (\ref{equation:lewis-001-prefinal}) is easily verified with (\ref{equation:H1Thm1.1}) and product rearrangements:
\begin{align}
\overline{J}_{4,64} \overline{J}_{8,32}\overline{J}_{28,64}
&= \overline{J}_{12,64}\overline{J}_{28,64}\overline{J}_{32,128}
+ q^4\overline{J}_{4,64}\overline{J}_{20,64}\overline{J}_{32,128}\\
&=\overline{J}_{32,128}( \overline{J}_{12,64}\overline{J}_{28,64}
+ q^4\overline{J}_{4,64}\overline{J}_{20,64})
=\overline{J}_{32,128}\overline{J}_{4,32}\overline{J}_{8,32}.\notag 
\end{align}

We prove identity (\ref{equation:lewis-000}).  Using (\ref{equation:J4-expansion}) and simplifying, (\ref{equation:lewis-001}) is then equivalent to 
{\allowdisplaybreaks \begin{align}
&\frac{\overline{J}_{16,32}}{J_{4,32}}\frac{J_{4,16}J_{8,32}\overline{J}_{32,128}}{\overline{J}_{12,64}}
- \overline{J}_{16,32}\overline{J}_{28,64}-q^4\overline{J}_{32,128}\overline{J}_{28,64}
+  q^4 \overline{J}_{8,32}\overline{J}_{52,64}\\
&\ \ \ \ \ =-q^8\frac{\overline{J}_{4,64}\overline{J}_{20,64}\overline{J}_{32,128}}
{\overline{J}_{12,64}},\notag
\end{align}}%
which is equivalent to 
\begin{align}
&\overline{J}_{16,32}J_{4,16}J_{8,32}\overline{J}_{32,128}
- \overline{J}_{16,32}\overline{J}_{28,64}J_{4,32}\overline{J}_{12,64}
-q^4\overline{J}_{32,128}\overline{J}_{28,64}J_{4,32}\overline{J}_{12,64}\label{equation:lewis-000-prefinal-id1}\\
&\ \ \ \ \ +  q^4 \overline{J}_{8,32}\overline{J}_{52,64}J_{4,32}\overline{J}_{12,64}
=-q^8\overline{J}_{4,64}\overline{J}_{20,64}\overline{J}_{32,128} J_{4,32}.\notag
\end{align}

Let us consider the pieces of (\ref{equation:lewis-000-prefinal-id1}).  Using (\ref{equation:H1Thm1.1}) and product rearrangements, we have
\begin{align}
q^4&\overline{J}_{32,128}\overline{J}_{28,64}J_{4,32}\overline{J}_{12,64}
-q^8\overline{J}_{4,64}\overline{J}_{20,64}\overline{J}_{32,128}J_{4,32}
\label{equation:lewis-000-prefinal-piece1}\\
&\ \ \ \ \ =q^4\overline{J}_{32,128}J_{4,32}(\overline{J}_{12,64}\overline{J}_{28,64}
-q^4\overline{J}_{4,64}\overline{J}_{20,64})
=q^4\overline{J}_{32,128}J_{4,32}J_{4,32}J_{8,32}.\notag 
\end{align}
Using product rearrangements, (\ref{equation:1.11}), and (\ref{equation:jsplit}), we obtain
{\allowdisplaybreaks \begin{align}
\overline{J}_{16,32}&J_{4,16}J_{8,32}\overline{J}_{32,128}-q^4\overline{J}_{32,128}J_{4,32}J_{4,32}J_{8,32}\notag \\
&=\overline{J}_{32,128}J_{8,32}
(\overline{J}_{16,32}J_{4,16}-q^4J_{4,32}^2)\notag \\
&=J_{8,64}J_{40,64}
(\overline{J}_{16,32}J_{4,16}-q^4J_{4,32}^2)\notag \\
&=J_{8,64}J_{40,64}
\Big (\frac{J_{4,32}^2}{J_{8,64}}\frac{\overline{J}_{4,32}J_{20,32}}{J_{16,64}}-q^4J_{4,32}^2\Big )\notag \\
&=J_{8,64}J_{40,64}
\Big (\frac{J_{4,32}^2}{J_{8,64}}j(-q^4;-q^{16})-q^4J_{4,32}^2\Big )
=J_{4,32}^2J_{24,64}^2.\label{equation:lewis-000-prefinal-piece2}
\end{align}}%
Using (\ref{equation:lewis-000-prefinal-piece1}), (\ref{equation:lewis-000-prefinal-piece2}), and factoring out a $J_{4,32}$, identity (\ref{equation:lewis-000-prefinal-id1}) is then equivalent to 
\begin{equation}
J_{4,32}J_{24,64}^2- \overline{J}_{16,32}\overline{J}_{28,64}\overline{J}_{12,64}+  q^4 \overline{J}_{8,32}\overline{J}_{52,64}\overline{J}_{12,64}=0.
\end{equation}
which is equivalent to
\begin{equation}
J_{24,64}^2J_{4,64}J_{36,64}
+  q^4 \overline{J}_{12,64}^2\overline{J}_{8,64}\overline{J}_{40,64}
=\overline{J}_{12,64}\overline{J}_{28,64}\overline{J}_{16,64}\overline{J}_{48,64},\label{equation:last-stop}
\end{equation}
where we have used (\ref{equation:1.10}) and factored out common terms.   Identity (\ref{equation:last-stop}) follows from (\ref{equation:Weierstrass}) with $q\mapsto q^{64}$ $a\mapsto -q^{32}$, $b\mapsto q^{20}$, $c\mapsto q^{16}$, $d\mapsto -q^{8}$.
\end{proof}

\begin{proof}[Proof of Proposition \ref{proposition:L-conj-ineq}]
The proofs for the three inequalities are the same, so we will only do the third one.  Instead of using Proposition \ref{proposition:L-conj}, we use the equivalent form identity (\ref{equation:L-conj-pre}), which gives
\begin{equation}
\sum_{n=0}^{\infty}\Big ( N(0,8;4n+3)-C(0,8;4n+3)\Big )q^n=q\cdot \frac{\overline{J}_{0,8}\overline{J}_{13,16}}{J_1}
=:\sum_{k=1}^{\infty}c(n)q^{n}.\label{equation:the-end}
\end{equation}
In sum form, the two theta functions in the numerator can be written
\begin{equation}
\overline{J}_{a,m}:=j(-q^a;q^m)=\sum_{n=-\infty}^{\infty}(-1)^nq^{m\binom{n}{2}}(-q^{a})^n
=\sum_{k=-\infty}^{\infty}q^{m\binom{k}{2}}q^{ak},
\end{equation}
which has only positive Fourier cofficients.  In product form, the denominator reads
\begin{equation}
J_1=J_{1,3}=\prod_{n=1}^{\infty}(1-q^n).
\end{equation}
By the geometric series $(1-q^n)^{-1}=\sum_{k\ge0}q^{nk}$, the denominator contributes only positive Fourier coefficients.  Given the lead factor $(1-q)^{-1}$, we know that every Fourier coefficient $c(n)$, $n\ge1$, of (\ref{equation:the-end}) will be strictly positive.
\end{proof}

\section*{Acknowledgements}
We would like to thank George Andrews for pointing out references \cite{LSG, NSG2}, Frank Garvan for pointing out reference \cite{L1}, and Bruce Berndt for helpful comments.


\begin{thebibliography}{999999}

\bibitem{AG} G. E. Andrews, F. G. Garvan, {\em Dyson's crank of a partition}, Bull. Amer. Math. Soc. (N.S.) {\bf 18} (1988), 167--171.

\bibitem{ABI} G. E. Andrews, B. C. Berndt,  {\em Ramanujan's Lost Notebook Part I}, Springer, New York, 2005.

\bibitem{ABCKM} G. E. Andrews, B. C. Berndt, S. H. Chan, S. Kim, A. Malik, {\em Four identities for third order mock theta functions}, preprint.

\bibitem{ASD} A. O. L. Atkin, H. P. F. Swinnerton-Dyer, {\em Some properties of partitions}, Proc. Lond. Math. Soc. (3), {\bf 4} (1954), 84--106. 

\bibitem{D} F. J. Dyson, {\em Some guesses in the theory of partitions}, Eureka, Cambridge {\bf 8} (1944), 10--15.

\bibitem{G1} F. G. Garvan,  {\em private communication}.

\bibitem{H1} D. R. Hickerson, {\em A proof of the mock theta conjectures}, Invent. Math. {\bf 94} (1988), no. 3, 639--660.

\bibitem{HM1} D. R. Hickerson, E. T. Mortenson, {\em Hecke-type double sums, Appell--Lerch sums, and mock theta functions, I},  Proc. Lond. Math. Soc. (3) {\bf 109} (2014), no. 2, 382--422. 

\bibitem{HM2} D. R. Hickerson, E. T. Mortenson, {\em Dyson's ranks and Appell--Lerch sums}, Math. Ann. {\bf 367} (2017), no. 1-2, 373--395. 

\bibitem{LSG} R. Lewis, N. Santa-Gadea, {\em On the rank and crank modulo $4$ and $8$}, Trans. Amer. Math. Soc. {\bf 341} (1994), 449--464.

\bibitem{L1} R. Lewis, {\em The generating functions of the rank and crank modulo $8$}, Ramanujan J. {\bf 18} (2009), 121--146.

\bibitem{M1} E. T. Mortenson, {\em On three third order mock theta functions and Hecke-type double-sums},  Ramanujan J. {\bf 30} (2013), no. 2, 279--308.

\bibitem{RLN} S. Ramanujan, {\em The Lost Notebook and Other Unpublished Papers}, Narosa Publishing House, New Delhi, 1988.

\bibitem{NSG2} N. Santa-Gadea, {\em On some relations for the rank moduli $9$ and $12$}, J. Number Theory {\bf 40} (1992), 130--145.

\bibitem{We} K. Weierstrass, {\em Zur Theorie der Jacobischen Funktionen von mehreren Ver\"anderlichen}, Sitzungsber. K\"onigl. Preuss. Akad. Wiss. (1882), 505--508; Werke band 3, pp. 155--159.

\end{thebibliography}
\end{document}